\documentclass[reqno,10pt,centertags]{amsart}
\usepackage{amsmath,amsthm,amscd,amssymb,latexsym,upref,enumerate,stmaryrd}
\usepackage{color,ulem}

\usepackage{hyperref}

\newcommand{\bbN}{{\mathbb{N}}}
\newcommand{\bbR}{{\mathbb{R}}}

\newcommand{\bbZ}{{\mathbb{Z}}}
\newcommand{\bbC}{{\mathbb{C}}}

\newcommand{\bbT}{{\mathbb{T}}}

\newcommand{\cB}{{\mathcal B}}

\newcommand{\cD}{{\mathcal D}}

\newcommand{\cF}{{\mathcal F}}

\newcommand{\cH}{{\mathcal H}}
\newcommand{\cI}{{\mathcal I}}

\newcommand{\cM}{{\mathcal M}}

\newcommand{\cS}{{\mathcal S}}

\newcommand{\cV}{{\mathcal V}}

\newcommand{\cX}{{\mathcal X}}

\newcommand{\gQ}{{\mathfrak{Q}}}
\newcommand{\ga}{{\mathfrak{a}}}
\newcommand{\gb}{{\mathfrak{b}}}

\newcommand{\no}{\notag}
\newcommand{\lb}{\label}
\newcommand{\bi}{\bibitem}
\newcommand{\ul}{\underline}
\newcommand{\ol}{\overline}

\newcommand{\wti}{\widetilde}
\newcommand{\hatt}{\widehat}

\newcommand{\f}{\frac}
\newcommand{\dott}{\,\cdot\,}

\renewcommand{\ge}{\geqslant}
\renewcommand{\le}{\leqslant}

\newcommand{\supp}{\operatorname{supp}}
\newcommand{\tr}{\operatorname{tr}}
\newcommand{\dom}{\operatorname{dom }}
\newcommand{\loc}{\operatorname{loc}}
\newcommand{\locunif}{\operatorname{loc \, unif}}
\newcommand{\ran}{\operatorname{ran}}
\DeclareMathOperator{\sgn}{sgn}

\DeclareMathOperator{\CH}{CH}

\newcommand{\abs}[1]{\lvert#1\rvert}
\newcommand{\norm}[1]{\left\Vert#1\right\Vert}

\newcommand{\Green}{{\mathcal G}}

\newcommand*{\mailto}[1]{\href{mailto:#1}{\nolinkurl{#1}}}
\newcommand{\arxiv}[1]{\href{http://arxiv.org/abs/#1}{arXiv:#1}}

\allowdisplaybreaks \numberwithin{equation}{section}

\newcommand{\dpl}{{}_{H^{-1}(\bbR)}\langle}
\newcommand{\dpr}{\rangle_{H^1(\bbR)}}

\newtheorem{theorem}{Theorem}[section]

\newtheorem{lemma}[theorem]{Lemma}
\newtheorem{corollary}[theorem]{Corollary}
\newtheorem{definition}[theorem]{Definition}
\newtheorem{hypothesis}[theorem]{Hypothesis}
\newtheorem{example}[theorem]{Example}
\theoremstyle{remark}
\newtheorem{remark}[theorem]{Remark}

\begin{document}

\title[The Spectral Problem for the Camassa--Holm Hierarchy]{Some Remarks on the Spectral Problem Underlying the Camassa--Holm Hierarchy}

\author[F.\ Gesztesy]{Fritz Gesztesy}
\address{Department of Mathematics,
University of Missouri, Columbia, MO 65211, USA}
\email{\mailto{gesztesyf@missouri.edu}}
\urladdr{\url{http://www.math.missouri.edu/personnel/faculty/gesztesyf.html}}

\author[R.\ Weikard]{Rudi Weikard}
\address{Department of Mathematics, University of
Alabama at Birmingham, Birmingham, AL 35294, USA}
\email{\mailto{rudi@math.uab.edu}}
\urladdr{\url{http://www.math.math.uab.edu/\~{}rudi/index.html}}

\dedicatory{Dedicated with great pleasure to Ludwig Streit on the occasion of 
his 75th birthday.}
\date{\today}

\begin{abstract}
We study particular cases of left-definite eigenvalue problems
$$
A \psi = \lambda B \psi, \, \text{ with $A \geq \varepsilon I$ for some $\varepsilon > 0$
and $B$ self-adjoint,}
$$
but $B$ not necessarily positive or negative
definite, applicable, in particular, to the eigenvalue problem underlying the Camassa--Holm
hierarchy. In fact, we will treat a more general version where $A$ represents a positive
definite Schr\"odinger or Sturm--Liouville operator $T$ in $L^2(\bbR; dx)$  associated with a
differential expression of the form $\tau = - (d/dx) p(x) (d/dx) + q(x)$, $x \in \bbR$, and
$B$ represents an operator of multiplication by $r(x)$ in $L^2(\bbR; dx)$, which, in
general, is not a weight, that is, it is not nonnegative (or nonpositive) a.e.\ on $\bbR$. In fact,
our methods naturally permit us to treat certain classes of distributions (resp.,  measures) for 
the coefficients $q$ and $r$ and hence considerably extend the scope of this (generalized) 
eigenvalue problem, without having to change the underlying Hilbert space $L^2(\bbR; dx)$. 
Our approach relies on rewriting the eigenvalue problem
$A \psi = \lambda B \psi$ in the form
$$
A^{-1/2} B A^{-1/2} \chi = \lambda^{-1} \chi, \quad \chi = A^{1/2} \psi,
$$
and a careful study of (appropriate realizations of) the operator $A^{-1/2} B A^{-1/2}$ in
$L^2(\bbR; dx)$.

In the course of our treatment, we review and employ various necessary and sufficient
conditions for $q$ to be relatively bounded (resp., compact) and relatively form bounded
(resp., form compact) with respect to $T_0 = - d^2/dx^2$ defined on $H^2(\bbR)$.
In addition, we employ a supersymmetric formalism which permits us to factor the second-order
operator $T$ into a product of two first-order operators familiar from (and inspired by) Miura's
transformation linking the KdV and mKdV hierarchy of nonlinear evolution equations. We also
treat the case of periodic coefficients $q$ and $r$, where $q$ may be a
distribution and $r$ generates a measure and hence no smoothness is assumed for $q$ and $r$.
\end{abstract}

\maketitle


\section{Introduction}
\label{s1}

In this paper we are interested in a particular realization of a generalized left-definite
spectral problem originally derived from the Camassa--Holm hierarchy of integrable
nonlinear evolution equations.

Before specializing to the one-dimensional context at hand, we briefly
address the notion of generalized spectral problems associated
with operator pencils of the type $A - z B$, $z\in\bbC$, for appropriate
densely defined and closed linear operators $A$ and $B$ in a complex,
separable Hilbert space $\cH$. As discussed in \cite[Sect.\ VII.6]{Ka80}, there are
several (and in general, inequivalent) ways to reformulate such generalized
spectral problems. For instance, if $B$ is boundedly invertible, one may consider the
spectral problem for the operators $B^{-1}A$ or $A B^{-1}$, and in some cases
(e.g., if $B \geq \varepsilon I_{\cH}$ for some $\varepsilon > 0$, a case also called
a right-definite spectral problem)
also that of $B^{-1/2} A B^{-1/2}$. Similarly, if $A$ is boundedly invertible, the
spectral problem for the linear pencil $A - z B$ can be reformulated in terms of
the spectral problems for $A^{-1} B$ or $B A^{-1}$, and sometimes (e.g., if
$A \geq \varepsilon I_{\cH}$ for some $\varepsilon > 0$, a case also called
a left-definite spectral problem)) in terms of that of $A^{-1/2} B A^{-1/2}$.

There exists an enormous body of literature for these kinds of generalized
spectral problems and without any possibility of achieving completeness, we
refer, for instance, to \cite{AADH94}, \cite{Be13}, \cite{FL86}, \cite{GGP10},
\cite{He87}, \cite{He85}, \cite{He86}, \cite{HK80}, \cite{LM08},
\cite{Pe68}, \cite{Tr00}, and the extensive literature cited therein in the context of general
boundary value problems. In the context of indefinite Sturm--Liouville-type boundary value
problems we mention, for instance, \cite{AM87}, \cite{Be85}, \cite{Be07}, \cite{BKT09},
\cite{BMT11}, \cite{BP10}, \cite{BT07}, \cite{BBW09}, \cite{BBW12}, \cite{BBW12a},
\cite{BE83}, \cite{BF09}, \cite{Co97}, \cite{Co97a}, \cite{Co98}, \cite{DL86}, \cite{Ec12},
\cite{ET12}, \cite{FRY02}, \cite{Ka10}, \cite{KM07}, \cite{KT09}, \cite{KWZ01},
\cite{KWZ04}, \cite{KWZM03}, \cite{MZ02}, \cite{MZ05}, \cite{Ph13}, \cite{Vo96},
\cite[Chs.\ 5, 11, 12]{Ze05}, and again no attempt at a comprehensive account of
the existing literature is possible due to the enormous volume of the latter.

The prime motivation behind our attempt to study certain left-definite eigenvalue problems
is due to their natural occurrence in connection with the Camassa--Holm (CH) hierarchy. For
a detailed treatment and an extensive list of references we refer to \cite{GH03a},
\cite[Ch.\ 5]{GH03} and \cite{GH08}. The first few equations of the CH hierarchy (cf., e.g.,
\cite[Sect.\ 5.2]{GH03} for a recursive approach to the CH hierarchy) explicitly read
(with $u = u(x,t)$, $(x,t) \in \bbR^2$)
\begin{align}
\CH_0(u)&=4u_{t_{0}}-u_{xxt_{0}}+u_{xxx}-4u_{x}=0, \no \\
\CH_1(u)&=4u_{t_{1}}-u_{xxt_{1}} -2u u_{xxx}-4u_{x}u_{xx}
+24u u_{x}+c_{1}(u_{xxx}-4u_{x})=0, \no \\
\CH_2(u)&=4u_{t_{2}}-u_{xxt_{2}}+2u^2 u_{xxx}-8u u_{x} u_{xx}
-40u^2 u_{x} \lb{1.1} \\
& \quad +2(u_{xxx}-4u_x)\Green\big(u_{x}^2+8u^2\big)
-8(4u-u_{xx})\Green\big(u_{x} u_{xx}+8u u_{x} \big) \no \\
&\quad +c_1(-2u u_{xxx}-4u_{x}u_{xx}+24u u_{x})
+c_{2}(u_{xxx}-4u_{x})=0, \, \text{  etc.,} \no
\end{align}
for appropriate constants $c_\ell$, $\ell\in\bbN$. Here $\Green$ is given by
\begin{equation}
\Green\colon \begin{cases} L^\infty(\bbR; dx)\to L^\infty(\bbR; dx), \\
\, v \mapsto (\Green v)(x)=\frac14 \int_{\bbR} dy\, e^{-2\abs{x-y}} v(y),
\quad x\in\bbR,
\end{cases}    \lb{1.2}
\end{equation}
and one observes that $\Green$ is the
resolvent of minus the one-dimensional Laplacian at energy parameter
equal to $-4$, that is,
\begin{equation}
    \Green=\bigg(-\frac{d^2}{dx^2}+4\bigg)^{-1}.     \lb{1.3}
\end{equation}

The spectral problem underlying
the CH hierarchy can then be cast in the form (with ``prime'' denoting $d/dx$),
\begin{equation}
\Phi'(z,x)=U(z,x)\Phi(z,x), \quad (z,x) \in \bbC \times \bbR,    \lb{1.4}
\end{equation}
where
\begin{align}
\begin{split}
\Phi(z,x)=\begin{pmatrix} \phi_1(z,x) \\ \phi_2(z,x)
\end{pmatrix}, \quad U(z,x)=\begin{pmatrix} -1 &1\\
z [u_{xx}(x) - 4u(x)] &1 \end{pmatrix},& \\
(z,x) \in\bbC \times \bbR.&    \lb{1.5}
\end{split}
\end{align}
Eliminating $\phi_2$ in \eqref{1.4} then results in the scalar (weighted) spectral problem
\begin{equation}
- \phi''(z,x) + \phi(z,x) = z [u_{xx}(x) - 4u(x)] \phi(z,x), \quad
(z,x) \in (\bbC\backslash\{0\}) \times \bbR.   \lb{1.6}
\end{equation}

In the specific context of the left-definite Camassa--Holm spectral problem we refer to 
\cite{BSS00}, \cite{BSS05}, 
\cite{Be04}, \cite{BBW12}, \cite{Co97}, \cite{Co97a}, \cite{Co98}, \cite{Co01}, \cite{CGI06}
\cite{CL03}, \cite{DL86}, \cite{ET13}, \cite{GH08}, \cite{KKM09}, \cite{Ko04}, \cite{Ko11},
\cite{Mc01}, \cite{Mc03}, \cite{Mc03a}, \cite{Mc04}, and the literature cited therein.

Rather than directly studying \eqref{1.6} in this note, we will study
some of its generalizations and hence focus on several spectral problems originating with the
general Sturm--Liouville equation
\begin{equation}
- (p(x) \psi'(z,x))' + q(x) \psi(z,x) = z r(x) \psi(z,x), \quad
(z,x) \in (\bbC\backslash\{0\}) \times \bbR,     \lb{1.7}
\end{equation}
under various hypotheses on the coefficients $p, q, r$ to be described in more detail later on
and with emphasis on the fact that $r$ may change its sign.
At this point we assume the following basic requirements on $p, q, r$ (but we emphasize that
later on we will consider vastly more general situations where $q$ and $r$ are permitted to
lie in certain classes of distributions):

\begin{hypothesis} \lb{h1.1}
$(i)$ Suppose that $p>0$ a.e.\ on $\bbR$, $p^{-1} \in L^1_{\loc}(\bbR; dx)$,
and that $q, r \in L^1_{\loc}(\bbR; dx)$ are real-valued a.e.\ on $\bbR$. In
addition, assume that $r \neq 0$ on a set of positive Lebesgue measure and that
\begin{equation}
\pm \lim_{x \to \pm \infty} \int^x dx' \, p(x')^{-1/2} = \infty.   \lb{1.8}
\end{equation}
$(ii)$ Introducing the differential expression
\begin{equation}
\tau = - \f{d}{dx} p(x) \f{d}{dx} + q(x), \quad x \in \bbR,    \lb{1.9}
\end{equation}
and the associated minimal operator $T_{\min}$ in $L^2(\bbR; dx)$ by
\begin{align}
& T_{\min} f = \tau f,   \no \\
& \, f \in \dom(T_{\min}) = \big\{g \in L^2(\bbR; dx) \,\big|\, g, (p g') \in AC_{\loc}(\bbR); \, \supp \, (g) \, \text{compact};    \lb{1.10} \\
& \hspace*{8.8cm}  \tau g \in L^2(\bbR; dx)\big\},     \no
\end{align}
we assume that for some $\varepsilon >0$,
\begin{equation}
T_{\min} \geq \varepsilon I_{L^2(\bbR; dx)}.   \lb{1.11}
\end{equation}
\end{hypothesis}

We note that our assumptions \eqref{1.8} and \eqref{1.11} imply that
$\tau$ is in the limit point case at $+\infty$ and $-\infty$ (cf., e.g., \cite{CG03},
\cite{Ge93}, \cite{Ha48}, \cite{Re51}). This permits one
to introduce the maximally defined self-adjoint operator $T$ in $L^2(\bbR; dx)$ associated
with $\tau$ by
\begin{align}
\begin{split}
& T f = \tau f,   \\
& \, f \in \dom(T) = \big\{g \in L^2(\bbR; dx) \,\big|\, g, (p g') \in AC_{\loc}(\bbR);
\, \tau g \in L^2(\bbR; dx)\big\}      \lb{1.12}
\end{split}
\end{align}
(where $AC_{\loc}(\bbR)$ denotes the set of locally absolutely continuous
functions on $\bbR$). In particular, $T$ is the closure of $T_{\min}$,
\begin{equation}
T = \ol{T_{\min}},   \lb{1.13}
\end{equation}
and hence also
\begin{equation}
T \geq \varepsilon I_{L^2(\bbR; dx)}.   \lb{1.14}
\end{equation}

\begin{remark} \lb{r1.2}
By a result proven in Yafaev \cite{Ya72}, if $p =1$ and $q \geq 0$ a.e.\ on $\bbR$, \eqref{1.14} holds for some $\varepsilon > 0$ if and only if there exist $c_0 > 0$ and $R_0 > 0$ such that for all $x\in\bbR$ and all $a \geq R_0$, 
\begin{equation}
\int_x^{x+a} dx \, q(x) \geq c_0.   \lb{1.15}
\end{equation}
If $p$ is bounded below by some $\varepsilon_0 > 0$ (which we may choose smaller than one), one has 
\begin{equation} 
\int_\bbR dx \big[p(x) |u'(x)|^2 + q(x) |u(x)|^2\big] \geq 
\varepsilon_0^{-1}\int_\bbR dx \big[|u'(x)|^2+q(x) |u(x)|^2\big].
\end{equation} 
Hence \eqref{1.15} is then still sufficient for \eqref{1.14} to hold.

We also note that Theorem\ 3 in \cite{Ya72} shows that $q\geq0$ is not
necessary for \eqref{1.14} to hold. In fact, if $q_2\geq0$, but 
$\int_a^{a+1} dx \, q_2(x) \leq c$ for all $a\in\bbR$, one finds
\begin{equation}
-(c+4c^2)\int_\bbR dx \, |u(x)|^2 \leq \int_\bbR dx \big[|u'(x)|^2 - q_2(x) |u(x)|^2\big].
\end{equation}
Hence if $p=1$ and $q=\varepsilon+c+4c^2-q_2$ one obtains \eqref{1.14} even though $q$ may assume negative values.
\end{remark}

Given these preparations, we now associate the weighted eigenvalue equation \eqref{1.7}
with a standard self-adjoint spectral problem of the form
\begin{align}
\begin{split}
& \ul{T^{-1/2} r T^{-1/2}} \, \chi = \zeta \chi,   \\
& \, \chi (\zeta,x) = \big(T^{1/2} \psi (z,\cdot)\big)(x),
\;\; \zeta = 1/z \in \bbC\backslash\{0\}, \; x \in \bbR,     \lb{1.16}
\end{split}
\end{align}
for the integral operator $\ul{T^{-1/2} r T^{-1/2}}$ in $L^2(\bbR; dx)$, subject to
certain additional conditions on $p, q, r$. We use the particular notation $\ul{T^{-1/2} r T^{-1/2}}$
to underscore the particular care that needs to be taken with interpreting this expression as a
bounded, self-adjoint operator in $L^2(\bbR; dx)$ (pertinent details can be found in \eqref{2.19a}
and, especially, in \eqref{3.74}). It is important to note that in contrast to a number of papers that find it
necessary to use different Hilbert spaces in connection with a left-definite spectral problem (in some cases the weight $r$ is replaced by $|r|$, in other situations the new Hilbert space is
coefficient-dependent), our treatment works with one and the same underlying Hilbert space
$L^2(\bbR; dx)$.

We emphasize that rewriting \eqref{1.7} in the form \eqref{1.16} is not new. In
particular, in the context of the CH spectral problem \eqref{1.6} this has briefly
been used, for instance, in \cite{CM99} (in the periodic case), in \cite{Co01}
(in the context of the CH scattering problem), in
\cite{GH08} (in connection with real-valued algebro-geometric CH solutions),
and in \cite{Mc03} (in connection with CH flows and Fredholm determinants).
However, apart from the approach discussed in \cite{BBW09}, \cite{BBW12}, \cite{BBW12a}, most
investigations associated with the CH spectral problem \eqref{1.6}
appear to focus primarily on certain Liouville--Green transformations which transform \eqref{1.6}
into a Schr\"odinger equation for some effective potential coefficient (see, e.g., \cite{Co97},
\cite{Co97a}, \cite{Co98}). This requires additional assumptions on the coefficients which
in general can be avoided in the context of \eqref{1.16}. Indeed, the change of variables
\begin{equation}
\bbR \ni x \mapsto t=\int_0^x dx' \, p(x')^{-1},      \lb{1.17}
\end{equation}
 turns the equation $-(pu')'+qu=zru$ on $\bbR$ into
\begin{align}
\begin{split}
& -v''+Q v=z R v \, \text{ on $\bigg(- \int_{-\infty}^0 dx' \, p(x')^{-1},
\int_0^{+\infty} dx' \, p(x')^{-1} \bigg)$},     \lb{1.18} \\
& \, v(t)=u(x(t)), \quad Q(t)=p(x(t))q(x(t)), \quad R(t)=p(x(t))r(x(t)).
\end{split}
\end{align}
However, assuming for instance, $\pm \int^{\pm \infty} dx \, p(x) = \infty$, the change of
variables is only unitary between the spaces $L^2(\bbR; dx)$ and $L^2(\bbR; dx/p(x))$ and
hence necessitates a change in the underlying measure.

The primary aim of this note is to sketch a few instances in which the integral operator approach
in \eqref{1.16} naturally, and in a straightforward manner, leads to much more general spectral
results and hence is preferable to the Liouville--Green approach. In particular, we are interested in
generalized situations, where the coefficients $q$ and $r$ lie in certain classes of distributions.
To the best of our knowledge, this level of generality is new in this context.

In Section \ref{s2} we analyze basic spectral theory of $\ul{T^{-1/2} r T^{-1/2}}$ in
$L^2(\bbR; dx)$ assuming Hypothesis \ref{h1.1} and appropriate additional assumptions on
$p, q, r$. The more general case where $q$ and $r$ lie in certain classes of distributions is
treated in detail in Section \ref{s3}. There we heavily rely on supersymmetric methods and Miura transformations. This approach exploits the intimate relationship between spectral theory for
Schr\"odinger operators factorized into first-order differential operators and that of an
associated Dirac-type operator. Section \ref{s4} is devoted to applications in the special case
where $q$ and $r$ are periodic (for simplicity we take $p=1$). We permit $q$ to lie in a class of distributions and $r$ to be a signed measure, which underscores the novelty of our approach. 
Three appendices provide ample background results: 
Appendix \ref{sA} is devoted to basic facts on relative boundedness and compactness of operators 
and forms; the supersymmetric formalism relating Schr\"odinger and Dirac-type operators is 
presented in Appendix \ref{sB}, and details on sesquilinear forms and their associated operators 
are provided in Appendix \ref{sC}.

Finally, we briefly summarize some of the notation used in this paper: Let $\cH$ be a separable
complex Hilbert space, $(\cdot,\cdot)_{\cH}$ the scalar product in
$\cH$ (linear in the second factor), and $I_{\cH}$ the identity operator in $\cH$.
Next, let $T$ be a linear operator mapping (a subspace of) a
Banach space into another, with $\dom(T)$, $\ran(T)$, and $\ker(T)$ denoting the
domain, range, and kernel (i.e., null space) of $T$. The closure of a closable
operator $S$ is denoted by $\ol S$.

The spectrum, essential spectrum, point spectrum, discrete spectrum, absolutely continuous
spectrum, and resolvent set of a closed linear operator in $\cH$ will be denoted by
$\sigma(\cdot)$, $\sigma_{\rm ess}(\cdot)$, $\sigma_{\rm p}(\cdot)$, $\sigma_{\rm d}(\cdot)$,
$\sigma_{\rm ac}(\cdot)$, and $\rho(\cdot)$, respectively.

The Banach spaces of bounded and compact linear operators in $\cH$ are
denoted by $\cB(\cH)$ and $\cB_\infty(\cH)$, respectively. Similarly,
the Schatten--von Neumann (trace) ideals will subsequently be denoted
by $\cB_s(\cH)$, $s\in (0,\infty)$. The analogous notation
$\cB(\cX_1,\cX_2)$, $\cB_\infty (\cX_1,\cX_2)$, etc., will be used for bounded and compact
operators between two Banach spaces $\cX_1$ and $\cX_2$.
Moreover, $\cX_1\hookrightarrow \cX_2$ denotes the continuous embedding
of the Banach space $\cX_1$ into the Banach space $\cX_2$. Throughout this manuscript we
use the convention that if $X$ denotes a Banach space, $X^*$ denotes the {\it adjoint space}
of continuous  conjugate linear functionals on $X$, also known as the {\it conjugate dual} of $X$.

In the bulk of this note, $\cH$ will typically
represent the space $L^2(\bbR; dx)$. Operators of multiplication by a
function $V\in L^1_{\loc}(\bbR; dx)$ in $L^2(\bbR; dx)$ will by a slight abuse of notation again be denoted by $V$ (rather than the frequently used, but more cumbersome, notation $M_V$)
and unless otherwise stated, will always assumed to be maximally defined in $L^2(\bbR; dx)$
(i.e., $\dom(V) = \big\{f \in L^2(\bbR; dx) \,\big|\, Vf \in L^2(\bbR; dx)\big\}$). Moreover, in
subsequent sections, the identity operator
$I_{L^2(\bbR; dx)}$ in $L^2(\bbR; dx)$ will simply be denoted by $I$ for brevity.

The symbol $\cD(\bbR)$ denotes the space of test functions $C_0^{\infty}(\bbR)$ with its usual
(inductive limit) topology. The corresponding space of continuous linear functionals on
$\cD(\bbR)$ is denoted by $\cD^{\prime}(\bbR)$ (i.e.,
$\cD^{\prime}(\bbR) = C_0^{\infty}(\bbR)^{\prime}$).

\section{General Spectral Theory of $\ul{T^{-1/2} r T^{-1/2}}$}
\lb{s2}

In this section we derive some general spectral properties of
$\ul{T^{-1/2} r T^{-1/2}}$ which reproduce some known results that were
originally derived in the CH context of \eqref{1.6}, but now we prove
them under considerably more general conditions on the coefficients
$p, q, r$, and generally, with great ease. In this section $p, q, r$ will satisfy
Hypothesis \ref{h1.1} and appropriate additional assumptions. (The case
where $q, r$ lie in certain classes of distributions will be treated in Section \ref{s3}.)

For a quick summary of the notions of relatively bounded and compact
operators and forms frequently used in this section, we refer to Appendix \ref{sA}.

Before analyzing the operator $\ul{T^{-1/2} r T^{-1/2}}$ we recall three
useful results:

We denote by $T_0$ (minus) the usual Laplacian in $L^2(\bbR; dx)$ defined by
\begin{align}
& T_0 f = - f'',   \lb{2.1} \\
& \, f \in \dom(T_0) = \big\{g \in L^2(\bbR; dx) \,\big|\, g, g' \in AC_{\loc}(\bbR);
\, g'' \in L^2(\bbR; dx)\big\} = H^{2}(\bbR),    \no
\end{align}
where $H^{m}(\bbR)$, $m \in \bbN$, abbreviate the usual Sobolev
spaces of functions whose distributional derivatives up to order $m$ lie in
$L^2(\bbR; dx)$.

In the following it is useful to introduce the spaces of locally uniformly $L^p$-integrable
functions on $\bbR$,
\begin{equation}
L^p_{\locunif}(\bbR; dx) = \bigg\{f \in L^p_{\loc}(\bbR; dx) \,\bigg|\,
\sup_{a\in\bbR} \bigg(\int_a^{a+1} dx \, |f(x)|^p\bigg) < \infty\bigg\}, \quad
p \in [1,\infty).      \lb{2.1a}
\end{equation}
Equivalently, let
\begin{equation}
\eta \in C_0^{\infty}(\bbR), \quad 0 \leq \eta \leq 1, \quad \eta|_{B(0; 1)} = 1,    \lb{2.1b}
\end{equation}
with $B(x;r) \subset \bbR$ the open ball centered at $x_0 \in \bbR$ and radius $r>0$, then
\begin{equation}
L^p_{\locunif}(\bbR; dx) = \Big\{f \in L^p_{\loc}(\bbR; dx) \,\Big|\,
\sup_{a\in\bbR} \| \eta( \cdot - a) f\|_{L^p(\bbR; dx)} < \infty\Big\}, \quad
p \in [1,\infty).      \lb{2.1c}
\end{equation}

We refer to Appendix \ref{sA} for basic notions in connection with relatively bounded linear
operators.

\begin{theorem} $($\cite[Theorem\ 2.7.1]{Sc81}, \cite[p.\ 35]{Si71}.$)$ \lb{t2.1}
Let $V, w \in L^2_{\loc} (\bbR; dx)$. \\
Then the following conditions $(i)$--$(iv)$ are
equivalent:
\begin{align}
& (i)  \;\;\, \dom(w) \supseteq \dom\big(T_0^{1/2}\big) = H^1(\bbR).    \lb{2.2} \\
& (ii) \,\,\, w \in L^2_{\locunif}(\bbR; dx).    \lb{2.3} \\
& (iii) \text{ For some $C>0$,}   \no \\
& \qquad \, \|w f\|_{L^2(\bbR; dx)}^2 \leq
C \Big[\big\| T_0^{1/2} f\|_{L^2(\bbR; dx)}^2
+  \|f\|_{L^2(\bbR; dx)}^2\Big],   \lb{2.3a} \\
& \hspace*{4.52cm}  f \in \dom\big(T_0^{1/2}\big) = H^{1}(\bbR).  \no \\
& (iv) \text{ For all $\varepsilon > 0$, there exists $C_{\varepsilon}>0$
such that:}   \no \\
& \qquad \, \|w f\|_{L^2(\bbR; dx)}^2 \leq \varepsilon
\big\| T_0^{1/2} f\big\|_{L^2(\bbR; dx)}^2
+ C_{\varepsilon} \|f\|_{L^2(\bbR; dx)}^2,   \lb{2.4} \\
& \hspace*{4.52cm}  f \in \dom\big(T_0^{1/2}\big) = H^{1}(\bbR).  \no
\intertext{Moreover, also the following conditions $(v)$--$(viii)$ are equivalent:}
& (v)  \;\;\, \dom(V) \supseteq \dom(T_0) = H^2(\bbR).    \lb{2.5} \\
& (vi) \,\,\, V \in L^2_{\locunif}(\bbR; dx).
\lb{2.6} \\
& (vii) \text{ For some $C>0$,}   \no \\
& \qquad \;\, \|V f\|_{L^2(\bbR; dx)}^2 \leq C \Big[\| T_0 f\|_{L^2(\bbR; dx)}^2
+ \|f\|_{L^2(\bbR; dx)}^2\Big],   \lb{2.6a} \\
& \hspace*{4.7cm}  f \in \dom(T_0) = H^{2}(\bbR).  \no \\
& (viii) \text{ For all $\varepsilon > 0$, there exists $C_{\varepsilon}>0$
such that:}   \no \\
& \qquad \;\;\, \|V f\|_{L^2(\bbR; dx)}^2 \leq \varepsilon
\| T_0 f\|_{L^2(\bbR; dx)}^2
+ C_{\varepsilon} \|f\|_{L^2(\bbR; dx)}^2,   \lb{2.7} \\
& \hspace*{4.75cm}  f \in \dom(T_0) = H^{2}(\bbR).  \no
\end{align}
In fact, it is possible to replace $T_0^{1/2}$ by any polynomial
$P_m\big(T_0^{1/2}\big)$ of degree $m\in\bbN$ in connection with
items $(i)$--$(iv)$.
\end{theorem}

We emphasize the remarkable fact that according to items $(iii)$, $(iv)$ and $(vii)$, $(viii)$,
relative form and operator boundedness is actually equivalent to {\it infinitesimal} form and
operator boundedness in Theorem \ref{t2.1}.

For completeness, we briefly sketch some of the principal ideas underlying items $(i)$--$(iv)$
in Theorem \ref{t2.1}, particularly, focusing on item $(ii)$:
That item $(i)$ implies item $(iii)$ is of course a consequence of the closed graph theorem.
Exploiting continuity of $f \in H^1(\bbR)$, yields for arbitrary $\varepsilon > 0$,
\begin{align}
|f(x)|^2 - |f(x')|^2 &= \int_{x'}^x dy \, \big[\ol{f(y} f'(y) + \ol{f'(y)} f(y)\big]     \\
& \leq \varepsilon \int_{\cI} dy \, |f'(y|^2 + \varepsilon^{-1} \int_{\cI} dy \, |f(y)|^2,  \quad
f \in H^{1}(\bbR), \; x, x' \in \cI,   \no 
\end{align}
with $\cI\subset\bbR$ an arbitrary interval of length one.
The use of the mean value theorem for integrals then permits one to choose $x' \in \cI$
such that
\begin{equation}
|f(x')|^2 = \int_{\cI} dy \, |f(y)|^2
\end{equation}
implying
\begin{equation}
|f(x)|^2 \leq \varepsilon \int_{\cI} dx' \, |f'(x')|^2 + (1+\varepsilon^{-1})
\int_{\cI} dx' \, |f(x')|^2, \quad f \in H^{1}(\bbR), \; x\in\cI,    \lb{2.7a}
\end{equation}
and hence after summing over all intervals $\cI$ of length one, and using boundedness of
$f \in H^1(\bbR)$,
\begin{equation}
|f(x)|^2 \leq\| f \|_{L^{\infty}(\bbR; dx)}^2 \leq \varepsilon \| f' \|_{L^2(\bbR; dx)}^2 +
(1 + \varepsilon^{-1}) \| f \|_{L^2(\bbR; dx)}^2, \quad f \in H^{1}(\bbR), \; x\in\bbR.
\end{equation}
Multiplying \eqref{2.7a} by $|w(x)|^2$ and integrating with respect to $x$ over $\cI$ yields
\begin{equation}
\int_{\cI} dx \, |w(x)|^2 |f(x)|^2 \leq \varepsilon C_0 \int_{\cI} dx' \, |f'(x')|^2
+ (1+\varepsilon^{-1}) C_0 \int_{\cI} dx' \, |f(x')|^2,
\end{equation}
and summing again over all intervals $\cI$ of length implies
\begin{equation}
\|w f\|_{L^2(\bbR; dx)}^2 \leq \varepsilon C_0 \|f'\|_{L^2(\bbR)}^2
+ \big(1 + \varepsilon^{-1}\big) C_0 \|f\|_{L^2(\bbR; dx)}^2, \quad
f \in H^{1}(\bbR),    \lb{2.8}
\end{equation}
where
\begin{equation}
C_0 := \sup_{a\in\bbR} \bigg(\int_a^{a+1} dx \, |w(x)|^2\bigg) < \infty,
\end{equation}
illustrating the sufficiency part of condition $w \in L^2_{\locunif}(\bbR; dx)$ in item $(ii)$ for
item $(iv)$ to hold.

Next, consider
$\psi(x) = e^{1-x^2}$, $\psi_a (x) = \psi(x-a)$, $x, a \in \bbR$. Then
\begin{align}
\int_a^{a+1} dx \, |w(x)|^2 & \leq \int_{\bbR} dx \, \big[|w(x)| |\psi_a(x)|\big]^2
\leq C \Big[\big\|T_0^{1/2} \psi_a\big\|^2_{L^2(\bbR; dx)}
+ \|\psi_a\|^2_{L^2(\bbR; dx)}\Big]  \\
& \leq \widehat C \big[\|\psi'\|_{L^2(\bbR; dx)}^2 + \|\psi\|_{L^2(\bbR; dx)}^2\big] = \widetilde C,
\end{align}
with $\widetilde C$ independent of $a$, illustrates necessity of the condition
$w \in L^2_{\locunif}(\bbR; dx)$ in item $(ii)$ for item $(iii)$ to hold.

Given $\varepsilon > 0$, there exists $\eta(\varepsilon) > 0$, such that the obvious inequality
\begin{align}
\begin{split}
\| f' \|_{L^2(\bbR; dx)}^2 \leq \varepsilon
\big\| T_0^{m/2} \big\|_{L^2(\bbR; dx)}^2 + \eta(\varepsilon) \| f \|_{L^2(\bbR; dx)}^2,& \\
f \in \dom\big(T_0^{m/2}\big), \; m\in\bbN, \; m\geq 2,&
\end{split}
\end{align}
holds. (It suffices applying the Fourier transform and using $|p| \leq \varepsilon |p|^{m}
+ \eta(\varepsilon)$, $m\in\bbN$, $m\geq 2$) to extend this
to polynomials in $T_0^{1/2}$. This illustrates the sufficiency of the condition
$V \in L^2_{\locunif}(\bbR; dx)$ in item $(vi)$ for item $(viii)$ to hold.

We note that items $(i)$--$(iv)$ in Theorem \ref{t2.1} are mentioned in
\cite[p.\ 35]{Si71} without proof, but the crucial hint that
$f \in H^{1}(\bbR)$ implies that $f \in AC_{\loc}(\bbR) \cap L^{\infty}(\bbR;dx)$,
is made there. We also remark that Theorem\ 2.7.1 in \cite{Sc81} is primarily concerned with
items $(v)$--$(viii)$ in Theorem \ref{t2.1}. Nevertheless, its method of proof
also yields the results \eqref{2.2}--\eqref{2.4}, in particular, it contains
the fundamental inequality \eqref{2.8}.

Next, we also recall the following result (we refer to Appendix \ref{sA} for details on the notion
of relative compactness for linear operators):

\begin{theorem} \lb{t2.2} $($\cite[Theorem\ 3.7.5]{Sc81}, \cite[Sects.\ 15.7, 15.9]{Sc02}.$)$ ${}$ \\
Let $w \in L^2_{\loc} (\bbR; dx)$. Then the following conditions $(i)$--$(iii)$ are
equivalent:
\begin{align}
& (i)  \;\;\, \text{$w$ is $T_0^{1/2}$-compact.}    \lb{2.10} \\
& (ii)  \,\,\, \text{$w$ is $T_0$-compact.}    \lb{2.11} \\
& (iii) \, \lim_{|a|\to\infty} \bigg(\int_a^{a+1} dx \, |w(x)|^2\bigg) = 0.
\lb{2.12}
\end{align}
In fact, it is possible to replace $T_0^{1/2}$ by any polynomial
$P_m\big(T_0^{1/2}\big)$ of degree $m\in\bbN$ in item $(i)$.
\end{theorem}

We note that $w \in L^2_{\loc}(\bbR; dx)$ together with condition \eqref{2.12} imply that
$w \in L^2_{\locunif}(\bbR; dx)$ (cf.\ \cite[p.\ 378]{Sc02}).

It is interesting to observe that the if and only if characterizations
\eqref{2.2}--\eqref{2.4} for relative (resp., infinitesimal) form boundedness
mentioned by Simon \cite[p.\ 35]{Si71}, and those in \eqref{2.10}--\eqref{2.12} for relative (form) compactness by Schechter in the first edition of
\cite[Sects.\ 15.7, 15.9]{Sc02}, were both independently published in 1971.

In the context of Theorems \ref{t2.1} and \ref{t2.2} we also refer to \cite{AH97}
for interesting results on necessary and sufficient conditions on relative
boundedness and relative compactness for perturbations of Sturm--Liouville
operators by lower-order differential expressions on a half-line (in addition, see \cite{BH04},
\cite{HM01a}).

We will also use the following result on trace ideals. To fix our notation, we
denote by $f(X)$ the operator of multiplication by the measurable function
$f$ on $\bbR$, and similarly, we denote by $g(P)$ the operator defined by
the spectral theorem for a measurable function $g$ (equivalently, the operator of multiplication by the measurable function $g$ in Fourier space
$L^2(\bbR; dp)$), where $P$ denotes the self-adjoint (momentum) operator defined by
\begin{equation}
P f = - i f', \quad \dom(P) = H^{1}(\bbR).      \lb{2.13}
\end{equation}

\begin{theorem} \lb{t2.3} $($\cite[Theorem\ 4.1]{Si05}.$)$ ${}$ \\
Let $f \in L^s(\bbR;dx)$, $g \in L^s(\bbR; dx))$, $s \in [2,\infty)$. Then
\begin{equation}
f(X) g(P) \in \cB_s \big(L^2(\bbR; dx)\big)     \lb{2.14}
\end{equation}
and
\begin{equation}
\|f(X) g(P) \|_{\cB_s (L^2(\bbR; dx))} \leq (2 \pi)^{-1/s} \|f\|_{L^s(\bbR; dx)}
\|g\|_{L^s(\bbR; dx)}.     \lb{2.15}
\end{equation}
If $s=2$, $f$ and $g$ are both nonzero on a set of positive Lebesgue measure, and
$f(X) g(P) \in \cB_2 \big(L^2(\bbR; dx)\big)$, then
\begin{equation}
f, g \in L^2(\bbR; dx).    \lb{2.16}
\end{equation}
\end{theorem}

Given these preparations, we introduce the following convenient assumption:

\begin{hypothesis} \lb{h2.4}
In addition to the assumptions in Hypothesis \ref{h1.1} suppose that the
form domain of $T$ is given by
\begin{equation}
\dom\big(T^{1/2}\big) = \dom\big(T_0^{1/2}\big) = H^{1}(\bbR).    \lb{2.17}
\end{equation}
\end{hypothesis}

Assuming for some positive constants $c$ and $C$ that
\begin{equation}
0 < c \leq p \leq C \, \text{ a.e.\ on $\bbR$,}    \lb{2.17a}
\end{equation}
an application of Theorem \ref{t2.1}\,$(i),(ii)$ shows that \eqref{2.17} holds if
$q\in L^1_{\loc}(\bbR; dx)$ satisfies
\begin{equation}
q \in L^1_{\locunif}(\bbR;dx).    \lb{2.17b}
\end{equation}
Indeed, since by hypothesis, $T$ is essentially self-adjoint on $\dom(T_{\min})$,
$T\geq \varepsilon I$ for some $\varepsilon > 0$, and $\dom\big(T^{1/2}\big) = H^1(\bbR)$, the
sesquilinear form $\gQ_{T}$ associated with $T$ is of the form
\begin{align}
\begin{split}
& \gQ_T (f,g) = \int_{\bbR} dx \, p(x) \ol{f'(x)} g'(x) + \int_{\bbR} dx \, q(x) \ol{f(x)} g(x),  \\
& f, g \in \dom(\gQ_T) = \dom\big(T^{1/2}\big) = \dom\big(T_0^{1/2}\big) = H^1(\bbR).
\end{split}
\end{align}
Hence, by Theorem \ref{t2.1}\,$(i), (ii)$, this is equivalent to \eqref{2.17b} keeping in mind that $q$ is such that \eqref{1.11} holds.

Our first result then reads as follows:

\begin{theorem} \lb{t2.5}
Assume Hypothesis \ref{h2.4}. \\
$(i)$ Then
\begin{align}
& |r|^{1/2}T^{-1/2} \in \cB\big(L^2(\bbR; dx)\big)    \lb{2.18} \\
\intertext{if and only if }
& r \in L^1_{\locunif}(\bbR; dx).    \lb{2.19}
\end{align}
In particular, if \eqref{2.19} holds, introducing
\begin{equation}
\ul{T^{-1/2} r T^{-1/2}} = \big[|r|^{1/2} T^{-1/2}\big]^* \sgn (r)
\big[|r|^{1/2} T^{-1/2}\big],      \lb{2.19a}
\end{equation}
one concludes that
\begin{equation}
\ul{T^{-1/2} r T^{-1/2}} \in \cB\big(L^2(\bbR; dx)\big).    \lb{2.20}
\end{equation}
$(ii)$ Let $r_0 \in \bbR$. Then
\begin{align}
&  |r - r_0|^{1/2}T^{-1/2} \in \cB_{\infty}\big(L^2(\bbR; dx)\big)    \lb{2.21} \\
\intertext{if and only if }
& \lim_{|a|\to\infty} \bigg(\int_a^{a+1} dx \, |r(x) - r_0|\bigg) = 0.    \lb{2.22}
\end{align}
In particular, if \eqref{2.22} holds, introducing
\begin{equation}
\ul{T^{-1/2} (r - r_0) T^{-1/2}} = \big[|r - r_0|^{1/2} T^{-1/2}\big]^* \sgn (r - r_0)
\big[|r - r_0|^{1/2} T^{-1/2}\big],      \lb{2.22a}
\end{equation}
one concludes that
\begin{equation}
\ul{T^{-1/2} (r - r_0) T^{-1/2}} \in \cB_{\infty}\big(L^2(\bbR; dx)\big).    \lb{2.23}
\end{equation}
$(iii)$ Let $r_0 \in \bbR$. Then
\begin{align}
&  |r - r_0|^{1/2}T^{-1/2} \in \cB_2 \big(L^2(\bbR; dx)\big)    \lb{2.24} \\
\intertext{if and only if }
& \int_{\bbR} dx \, |r(x) - r_0| < \infty.    \lb{2.25}
\end{align}
In particular, if \eqref{2.25} holds, then
\begin{equation}
\ul{T^{-1/2} (r - r_0) T^{-1/2}} \in \cB_1 \big(L^2(\bbR; dx)\big).    \lb{2.26}
\end{equation}
\end{theorem}
\begin{proof}
$(i)$ By hypothesis \eqref{2.17} and the closed graph theorem one concludes that
\begin{equation}
\big[(T_0 + I)^{1/2} T^{-1/2}\big] \in \cB\big(L^2(\bbR; dx)\big).      \lb{2.27}
\end{equation}
\big(and analogously, $ \big[(T_0 + I)^{1/2} T^{-1/2}\big]^{-1}
= T^{1/2} (T_0 + I)^{-1/2} \in \cB\big(L^2(\bbR; dx)\big)$\big).
The equivalence of \eqref{2.18} and \eqref{2.19} then follows from
\eqref{2.2} and \eqref{2.3} and the fact that
\begin{equation}
|r|^{1/2} T^{-1/2} = \big[|r|^{1/2} (T_0 + I)^{-1/2}\big]
\big[(T_0 + I)^{1/2} T^{-1/2}\big].     \lb{2.28}
\end{equation}
The inclusion \eqref{2.20} immediately follows from \eqref{2.18} and \eqref{2.19a}. \\
$(ii)$ The equivalence of \eqref{2.21} and \eqref{2.22} follows from
\eqref{2.10} and \eqref{2.12}. The inclusion \eqref{2.23} then follows from
\eqref{2.22a}, \eqref{2.27}, and \eqref{2.28} with $r$ replaced by $r - r_0$. \\
$(iii)$ The equivalence of \eqref{2.24} and \eqref{2.25} follows from
\eqref{2.15} and \eqref{2.16}, employing again \eqref{2.27} and the fact that
$(|p|^2 + 1)^{-1/2} \in L^2(\bbR; dp)$. The relation \eqref{2.26} once more follows from
\eqref{2.22a}, \eqref{2.27}, and \eqref{2.28} with $r$ replaced by $r - r_0$, and the
fact that $S \in \cB_1 (\cH)$ if and only if $|S| \in \cB_1 (\cH)$ and hence if and only if
$|S|^{1/2} \in \cB_2 (\cH)$.
\end{proof}

In the following we use the obvious notation for subsets of $\cM \subset \bbR$
and constants $c \in \bbR$:
\begin{equation}
c \, \cM = \{c \, x \in \bbR \,|\, x \in \cM\}.    \lb{2.30}
\end{equation}

\begin{corollary} \lb{c2.6}
Assume Hypothesis \ref{h2.4}. \\
$(i)$ If \eqref{2.22} holds for some $r_0 \in \bbR$, then
\begin{equation}
\sigma_{\rm ess} \big(\ul{T^{-1/2} r T^{-1/2}}\big)
= \begin{cases} r_0 \sigma_{\rm ess} \big(T^{-1}\big),
& r_0 \in \bbR \backslash \{0\},    \\
\{0\}, & r_0 = 0.
\end{cases}      \lb{2.31}
\end{equation}
$(ii)$ If \eqref{2.25} holds for some $r_0 \in \bbR$, then
\begin{equation}
\sigma_{\rm ac} \big(\ul{T^{-1/2} r T^{-1/2}}\big)
= \begin{cases} r_0 \sigma_{\rm ac} \big(T^{-1}\big),
& r_0 \in \bbR \backslash \{0\},    \\
\emptyset, & r_0 = 0.
\end{cases}      \lb{2.32}
\end{equation}
\end{corollary}
\begin{proof}
For $r_0 \in \bbR \backslash \{0\} $ it suffices to use the decomposition
\begin{equation}
\ul{T^{-1/2} r T^{-1/2}} = \ul{T^{-1/2} [r_0 + (r - r_0)]T^{-1/2}}
= r_0 T^{-1} + \ul{T^{-1/2} (r - r_0) T^{-1/2}}     \lb{2.33}
\end{equation}
and employ \eqref{2.23} together with Weyl's theorem (cf., e.g.,
\cite[Sect.\ IX.2]{EE89}, \cite[Sect.\ XIII.4]{RS78}, \cite[Sect.\ 9.2]{We80}) to obtain
\eqref{2.31}, and combine \eqref{2.26} and the Kato--Rosenblum theorem
(cf., e.g., \cite[Sect.\ X.3]{Ka80}, \cite[Sect.\ XI.3]{RS79}, \cite[Sect.\ 11.1]{We80})
to obtain \eqref{2.32}.

In the case $r_0 = 0$ relation \eqref{2.31} holds since $\ul{T^{-1/2} r T^{-1/2}}
\in \cB_{\infty}\big(L^2(\bbR; dx)\big)$ and $L^2(\bbR; dx)$ is infinite-dimensional.
By the same argument one obtains \eqref{2.32} for $r_0 = 0$.
\end{proof}

In connection with \eqref{2.31} we also recall that by the spectral mapping theorem 
for self-adjoint operators $A$ in $\cH$,
\begin{equation}
0 \neq z \in \sigma_{\rm ess} \big((A - z_0 I_{\cH})^{-1}\big), \; z_0 \in \rho(A),  \,
\text{ if and only if } \, z^{-1} + z_0 \in \sigma_{\rm ess} (A)     \lb{2.34}
\end{equation}
(cf., e.g., \cite[Sect.\ XIII.4]{RS78}). Finally, we mention that there exists
a large body of results on determining essential and absolutely continuous
spectra for Sturm--Liouville-type operators $T$ associated with the
differential expressions of the type $\tau = - \f{d}{dx} p(x) \f{d}{dx} + q(x)$, $x \in \bbR$.
We refer, for instance, to \cite[XIII.7]{DS88}, \cite[Chs.\ 2, 4]{MP81},
\cite[Sect.\ 24]{Na68}, and the literature cited therein.

\begin{remark} \lb{r2.7}
While it is well-known that for $T$ densely defined and closed in $\cH$,
\begin{align}
\begin{split}
& \text{$T$ is bounded (resp., compact, Hilbert--Schmidt)}    \\
& \quad \text{if and only if $T^* T$ is bounded (resp., compact, trace class),}
\end{split}
\end{align}
the following example due to G.\ Teschl \cite{Te10} shows that if $S$ is
bounded and self-adjoint in $\cH$ with spectrum $\sigma (S) = \{-1, 1\}$ then
\begin{equation}
\text{$T$ bounded is {\bf not} equivalent to $\ol{T^* S T}$ bounded}
\end{equation}
assuming $T^* S T$ to be densely defined in $\cH$ (and hence closable
in $\cH$, since $T^* ST$ is symmetric). Indeed, considering
\begin{equation}
T = \begin{pmatrix} A & 0 \\ 0 & A^{-1} \end{pmatrix}, \quad A= A^*, \quad
A \geq I_{\cH}, \quad S = \begin{pmatrix} 0 & I_{\cH} \\ I_{\cH} & 0 \end{pmatrix},
\end{equation}
then
\begin{equation}
\ol{T^* S T} = S,
\end{equation}
and hence $\ol{T^*ST}$ is bounded, but $T$ is unbounded if $A$ is chosen to be unbounded.

Thus one cannot assert on abstract grounds that
\begin{equation}
\ul{T^{-1/2} r T^{-1/2}} = \big[|r|^{1/2} T^{-1/2}\big]^* \sgn(r) |r|^{1/2} T^{-1/2}
\end{equation}
is bounded if and only if $|r|^{1/2} T^{-1/2}$ is. In fact, this is utterly wrong as we shall
discuss in the following Section \ref{s3}. Indeed, focusing directly on $|r|^{1/2} T^{-1/2}$ instead
of $\ul{T^{-1/2} r T^{-1/2}}$ ignores crucial oscillations of $r$ that permit one to considerably
enlarge the class of admissible weights $r$. In particular, thus far we relied on estimates of the
type
\begin{equation}
\big\||q|^{1/2} f \big\|_{L^2(\bbR; dx)}^2 \leq C
\Big[\big\|T_0^{1/2} f\big\|_{L^2(\bbR; dx)}^2 + \|f\|_{L^2(\bbR; dx)}^2\Big],
\quad f \in H^1(\bbR),
\end{equation}
equivalently,
\begin{equation}
\int_{\bbR} dx \, |q(x)| |f(x)|^2 \leq
\big\|\big[T_0 + I \big]^{1/2} f\big\|^2_{L^2(\bbR; dx)},
\quad f \in H^1(\bbR).
\end{equation}
Consequently, we ignored all oscillations of $q$ (and hence, $r$). Instead, we should focus on
estimating
\begin{equation}
\bigg| \int_{\bbR} dx \, q(x) |f(x)|^2 \bigg| \leq
\big\|\big[T_0 + I \big]^{1/2} f\big\|_{L^2(\bbR; dx)}^2, \quad f \in H^1(\bbR),
\end{equation}
and this will be the focus of the next Section \ref{s3}.
\end{remark}

\section{Distributional Coefficients} \lb{s3}

In this section we extend our previous considerations where $q, r \in L^1_{\locunif}(\bbR; dx)$,
to the case where $q$ and $r$ are permitted to lie in a certain class of distributions. The
extension to distributional coefficients will be facilitated by employing supersymmetric methods
and an underlying Miura transformation. This approach permits one to relate spectral theory for
Schr\"odinger operators factorized into a product of first-order differential operators with that
of an associated Dirac-type operator.

We start with some background (cf., e.g., \cite[Chs.\ 4--6]{Gr09}, \cite[Chs.\ 2, 3, 11]{MS09},
\cite[Ch.\ 3]{Mc00}) and fix our notation in
connection with Sobolev spaces. Introducing
\begin{equation}
L^2_s(\bbR) = L^2 \Big(\bbR; \big(1 + |p|^2\big)^s dp\Big), \quad s \in \bbR,
\end{equation}
and identifying,
\begin{equation}
L^2_0(\bbR) = L^2 (\bbR; dp) = \big(L^2 (\bbR; dp)\big)^* = \big(L^2_0(\bbR)\big)^*,
\end{equation}
one gets the chain of Hilbert spaces with respect to the pivot space
$L^2_0(\bbR) = L^2 (\bbR; dp)$,
\begin{equation}
L^2_s(\bbR) \subset L^2 (\bbR; dp) \subset L^2_{-s} (\bbR) = \big(L^2_s (\bbR)\big)^*, \quad s > 0.
\end{equation}
Next, we introduce the maximally defined operator $G_0$ of multiplication by the function
$\big(1 + |\cdot|^2\big)^{1/2}$ in $L^2 (\bbR; dp)$,
\begin{align}
\begin{split}
& (G_0 f)(p) = \big(1 + |p|^2\big)^{1/2} f(p), \\
& \, f \in \dom(G_0) = \Big\{ g \in L^2 (\bbR; dp) \, \Big| \,
\big(1 + |\cdot|^2\big)^{1/2} g \in L^2 (\bbR; dp)\Big\}.
\end{split}
\end{align}
The operator $G_0$ extends to an operator defined on the entire scale $L^2_s(\bbR)$,
$s \in \bbR$, denoted by $\wti G_0$, such that
\begin{equation}
\wti G_0 : L^2_s(\bbR) \to L^2_{s-1}(\bbR), \quad
\big(\wti G_0\big)^{-1} : L^2_s(\bbR) \to L^2_{s+1}(\bbR), \, \text{ bijectively, } \, \; s \in \bbR.
\end{equation}
In particular, while
\begin{equation}
I: L^2(\bbR; dp) \to \big(L^2(\bbR; dp)\big)^* = L^2(\bbR; dp)
\end{equation}
represents the standard identification operator between $L^2_0(\bbR) = L^2 (\bbR; dp)$ and its adjoint space, $\big(L^2 (\bbR; dp)\big)^* = \big(L^2_0(\bbR)\big)^*$, via Riesz's lemma, we emphasize that we will not identify $\big(L^2_s (\bbR)\big)^*$ with $L^2_s(\bbR)$ when $s>0$. In fact, it is the operator $\wti G_0^2$ that provides a unitary map
\begin{equation}
\wti G_0^2 : L^2_s(\bbR) \to L^2_{s-2}(\bbR), \quad s \in \bbR.
\end{equation}
In particular,
\begin{equation}
\wti G_0^2 : L^2_1(\bbR) \to L^2_{-1}(\bbR) = \big(L^2_1(\bbR)\big)^* \, \text{ is a unitary map},  \lb{3.10}
\end{equation}
and we refer to \eqref{C.37} for an abstract analog of this fact. 

Denoting the Fourier transform on $L^2(\bbR; dp)$ by $\cF$, and then extended to the entire scale
$L^2_s(\bbR)$, $s \in \bbR$, more generally, to $\cS'(\bbR)$ by $\wti \cF$ (with
$\wti \cF : \cS'(\bbR) \to \cS'(\bbR)$ a homeomorphism), one obtains the scale of Sobolev
spaces via
\begin{equation}
H^s(\bbR) = \wti \cF L^2_s(\bbR), \quad s \in \bbR, \quad L^2(\bbR; dx) = \cF L^2(\bbR; dp),
\lb{3.11}
\end{equation}
and hence,
\begin{align}
\cF G_0 \cF^{-1} &= (T_0 + I)^{1/2} : H^1(\bbR) \to L^2(\bbR; dx), \, \text{ bijectively, }     \lb{3.12} \\
\wti \cF \wti G_0 \wti \cF^{-1} &= \big(\wti T_0 + \wti I \big)^{1/2} : H^s(\bbR) \to H^{s-1}(\bbR),
\, \text{ bijectively, } \, \; s \in \bbR,    \lb{3.13} \\
\wti \cF \big(\wti G_0\big)^{-1} \wti \cF^{-1} &= \big(\wti T_0 + \wti I \big)^{-1/2} : H^s(\bbR)
\to H^{s+1}(\bbR), \, \text{ bijectively, } \, \;  s \in \bbR.    \lb{3.14}
\end{align}

We recall that $T_0$ was defined as
\begin{equation}
T_0 = - d^2/dx^2, \quad \dom(T_0) = H^2(\bbR),    \lb{3.14a}
\end{equation}
in \eqref{2.1}, but now the extension $\wti T_0$ of $T_0$ is defined on the entire Sobolev
scale according to \eqref{3.13},
\begin{equation}
\big(\wti T_0 + \wti I\big) : H^s(\bbR) \to H^{s-2}(\bbR) \, \text{ is a unitary map, } \; s \in \bbR,
\lb{3.14b}
\end{equation}
and the special case $s=1$ again corresponds to \eqref{C.37},
\begin{equation}
\big(\wti T_0 + \wti I\big) : H^1(\bbR) \to H^{-1}(\bbR) = \big(H^1(\bbR)\big)^* \,
\text{ is a unitary map.}    \lb{3.14c}
\end{equation}
In addition, we note that
\begin{align}
& H^0(\bbR) = L^2(\bbR; dx), \quad \big(H^s(\bbR)\big)^* = H^{-s}(\bbR), \quad s \in \bbR,  \\
\begin{split}
& \, \cS(\bbR) \subset H^s(\bbR) \subset H^{s'}(\bbR) \subset L^2(\bbR; dx) \subset
H^{-s'}(\bbR) \subset H^{-s}(\bbR) \subset \cS'(\bbR),     \\
& \hspace*{9.07cm} s > s' > 0.
\end{split}
\end{align}
Moreover, we recall that $H^s(\bbR)$ is conveniently and alternatively introduced as the
completion of $C_0^{\infty}(\bbR)$ with respect to the norm $\|\cdot\|_s$,
\begin{equation}
H^s(\bbR) = \ol{C_0^{\infty}(\bbR)}^{\|\cdot\|_{s}}, \quad s \in \bbR,
\end{equation}
where for $\psi \in C_0^{\infty}(\bbR)$ and $s\in\bbR$,
\begin{equation}
\|\psi\|_{s} = \bigg(\int_{\bbR} d\xi \, \big(1+\abs{\xi}^{2s}\big)  |\hatt \psi(\xi)|^2)\bigg)^{1/2}, \quad
\hatt \psi(\xi) = (2 \pi)^{-1/2} \int_{\bbR} dx \, e^{-i \xi x} \psi(x).
\end{equation}
Equivalently,
\begin{equation}
H^{s}(\bbR) = \bigg\{u\in \cS^\prime(\bbR)\,\bigg|\,
\norm{u}_{H^{s}(\bbR)}^2 = \int_{\bbR^n}d\xi \, \big(1+\abs{\xi}^{2s}\big)
|\hatt u(\xi)|^2 < \infty \bigg\}, \quad s \in \bbR.
\end{equation}
Similarly,
\begin{align}
\begin{split}
H_{\loc}^s(\bbR) &= \big\{u \in \cD^{\prime}(\bbR) \, \big| \, \|\psi \, u\|_{H^s(\bbR)} < \infty \text{ for all }
\psi \in C_0^{\infty}(\bbR)\big\}   \\
&= \big\{u \in \cD^{\prime}(\bbR) \, \big| \, \|\eta( \cdot - a) \, u\|_{H^s(\bbR)} < \infty
\text{ for all $a \in \bbR$}\big\}, \quad s \in \bbR
\end{split}
\end{align}
(cf.\ \cite[p.\ 140]{Gr09}), and
\begin{equation}
H_{\locunif}^s(\bbR) = \Big\{u \in H_{\loc}^s(\bbR) \,\Big| \,
\sup_{a\in\bbR} \| \eta( \cdot - a) u\|_{H^s(\bbR)} < \infty \Big\}, \quad
s \in \bbR,
\end{equation}
with $\eta$ defined in \eqref{2.1b}.

Moreover, as proven in \cite[Sect.\ 2]{DM10} (cf.\ also \cite{KPST05}, \cite{MM08}, \cite{MM09})  elements $q \in H^{-1}_{\loc}(\bbR) \subset \cD'(\bbR)$ can be represented by
\begin{equation}
q = q_2' \, \text{ for some } \, q_2 \in L^2_{\loc}(\bbR; dx).
\end{equation}
Similarly, if $q \in H^{s-1}(\bbR)$ for some $s\geq 0$, \cite[Lemma\ 2.1]{KPST05} proves the
representation
\begin{equation}
q= v_{\infty} + v_s^{\; \prime} \, \text{ for some } v_{\infty} \in H^{\infty}(\bbR), \;
v_s \in H^s(\bbR),
\end{equation}
where
\begin{equation}
H^{\infty}(\bbR) = \bigcap_{t \geq 0} H^t(\bbR) \subset C^{\infty}(\bbR).
\end{equation}
In particular, if $q \in H^{-1}(\bbR)$ one has the representation
\begin{equation}
q= v_{\infty} + q_2^{\; \prime} \, \text{ for some } v_{\infty} \in H^{\infty}(\bbR), \;
q_2 \in L^2(\bbR; dx).
\end{equation}
Next, for $q \in H^{-1}_{\locunif}(\bbR)$, \cite[Theorem\ 2.1]{HM01} proves the representation
\begin{equation}
q = q_1 + q_2' \, \text{ for some } \, q_j \in L_{\locunif}^j (\bbR; dx), \; j=1,2.   \lb{3.27}
\end{equation}
The decomposition $q = q_1 + q_2'$ in \eqref{3.13} is nonunique. In fact, also the representation
\begin{equation}
q = q_{\infty} + q_2' \, \text{ for some } \, q_{\infty} \in L^{\infty} (\bbR; dx), \;
q_2 \in L_{\locunif}^2 (\bbR; dx)    \lb{3.28}
\end{equation}
is proved in \cite[Theorem\ 2.1]{HM01}. Finally, if $q \in H^{-1}_{\loc}(\bbR)$ is periodic with
period $\omega > 0$, \cite[Remark\ 2.3]{HM01} (see also \cite[Proposition\ 1]{DM10}) provides
the representation
\begin{equation}
q = c + q_2' \, \text{ for some } \, c \in \bbC, \; q_2 \in L_{\locunif}^2 (\bbR; dx),
\, \text{ $q_2$ periodic with period $\omega > 0$.}    \lb{3.29}
\end{equation}

Next, we turn to sequilinear forms $\gQ_q$ generated by a distribution $q \in \cD^{\prime}(\bbR)$
as follows: For $f, g \in C_0^{\infty}(\bbR)$, $\ol f$ is a multiplier for $q$, that is,
$\ol{f} q = q \ol{f} \in \cD'(\bbR)$
and hence the distributional pairing 
\begin{equation}
{}_{\cD'(\bbR)}\langle q f, g\rangle_{\cD(\bbR)} = (f q)(g) = q(\ol{f}g) = \gQ_q (f,g), \quad
f, g \in C_0^\infty(\bbR),       \lb{3.30a}
\end{equation}
is well-defined and thus determines a sesquilinear form $\gQ_q (\cdot,\cdot)$ defined on
$\cD(\bbR) = C_0^{\infty}(\bbR)$. The distribution $q\in \cD'(\bbR)$ is called a {\it multiplier}
from $H^1(\bbR)$ to $H^{-1}(\bbR)$ if \eqref{3.30a} continuously extends from
$C_0^\infty(\bbR)$ to $H^1(\bbR)$, that is, for some $C>0$,
\begin{equation}
|\gQ_q (f,g)| \leq C \|f\|_{H^1(\bbR)} \|g\|_{H^1(\bbR)}, \quad f, g \in C_0^{\infty}(\bbR),  \lb{3.31a}
\end{equation}
and hence one defines this extension $\wti \gQ_q$ via
\begin{align}
\begin{split}
& \wti \gQ_q (f,g) = \lim_{n\to\infty} \gQ_q (f_n,g_n), \quad f, g \in H^1(\bbR),
\; f_n, g_n \in C_0^{\infty}(\bbR), \\
& \quad \text{assuming } \, \lim_{n\to\infty} \|f - f_n\|_{H^1(\bbR)} =0, \;
\lim_{n\to\infty} \|g - g_n\|_{H^1(\bbR)} =0.    \lb{3.32a}
\end{split}
\end{align}
(This extension is independent of the particular choices of sequences $f_n, g_n$ and
by polarization, \eqref{3.31a} for $f=g$ suffices to yield the extension $\wti \gQ_q$ in
\eqref{3.32a}.) The set of all multipliers from $H^1(\bbR)$ to $H^{-1}(\bbR)$ is usually
denoted by $M\big(H^1(\bbR), H^{-1}(\bbR)\big)$, equivalently, one could use the symbol
$\cB\big(H^1(\bbR), H^{-1}(\bbR)\big)$, the bounded linear operators mapping $H^1(\bbR)$
into $H^{-1}(\bbR)$. Thus, for $q \in M\big(H^1(\bbR), H^{-1}(\bbR)\big)$, the
distributional pairing \eqref{3.30a} extends to
\begin{equation}
{}_{H^{-1}(\bbR)}\langle q f, g\rangle_{H^1(\bbR)} = \wti \gQ_q (f,g), \quad
f, g \in H^1(\bbR).
\end{equation}

\begin{theorem} \lb{t3.1} $($\cite{BS02}, \cite[Sect.\ 2.5]{MS09}, \cite{MV05}, \cite{NS06}.$)$
Assume that $q \in \cD^{\prime}(\bbR)$ generates the sesquilinear form $\gQ_q$ as in \eqref{3.30a}.
Then the following conditions $(i)$--$(iii)$ are equivalent: \\[1mm]
$(i)$ $q$ is form bounded with respect to $T_0$, that is, for some $C>0$,
\begin{equation}
|\gQ_q(f,f)| \leq C \|f\|_{H^1(\bbR)}^2
= C\big[\|f'\|_{L^2(\bbR; dx)}^2 + \|f\|_{L^2(\bbR; dx)}^2\big], \quad f \in C_0^{\infty}(\bbR),  \lb{3.34a}
\end{equation}
equivalently,
\begin{equation}
q \in M\big(H^1(\bbR), H^{-1}(\bbR)\big).    \lb{3.35a}
\end{equation}
$(ii)$ $q$ is infinitesimally form bounded with respect to $T_0$, that is, for all $\varepsilon > 0$,
there exists $C_{\varepsilon}>0$, such that,
\begin{equation}
|\gQ_q(f,f)| \leq \varepsilon \|f'\|_{L^2(\bbR; dx)}^2
+ C_\varepsilon \|f\|_{L^2(\bbR; dx)}^2, \quad f \in H^1(\bbR).    \lb{3.36a}
\end{equation}
$(iii)$ $q$ is of the form
\begin{equation}
q = q_1 + q_2', \, \text{ where } \, q_j \in L^j_{\locunif}(\bbR; dx), \; j=1,2.    \lb{3.37a}
\end{equation}
Equivalently,
\begin{equation}
q \in H^{-1}_{\locunif}(\bbR).      \lb{3.38a}
\end{equation}
\end{theorem}

Of course, if \eqref{3.34a} (equivalently, \eqref{3.36a}) holds, it extends to $\wti \gQ_q$ and
all $f \in H^1(\bbR)$.

\begin{theorem} \lb{t3.2} $($\cite{MS09}, \cite{MV05}.$)$  Assume that $q \in \cD^{\prime}(\bbR)$.
Then the following conditions $(i)$ and $(ii)$ are equivalent: \\[1mm]
$(i)$ $q$ is form compact with respect to $T_0$, that is, the map
\begin{equation}
q \colon H^1(\bbR) \to H^{-1}(\bbR) \, \text{ is compact.}    \lb{3.35}
\end{equation}
$(ii)$ $q$ is of the form
\begin{equation}
q = q_1 + q_2', \, \text{ where } \, q_j \in L^j_{\locunif}(\bbR; dx), \; j=1,2,   \lb{3.15}
\end{equation}
and
\begin{equation}
\lim_{|a|\to\infty} \bigg(\int_a^{a+1} dx \, |q_1(x)|\bigg) = 0, \;
\lim_{|a|\to\infty} \bigg(\int_a^{a+1} dx \, |q_2(x)|^2\bigg) = 0.    \lb{3.16}
\end{equation}
\end{theorem}

The proof relies on Bessel capacity methods described in Maz'ya and Verbitsky \cite{MV02},
\cite{MV02a}, \cite{MV05}, \cite{MV06}, and Maz'ya and Shaposhnikova \cite{MS09}.

\begin{remark} \lb{r3.3}
If $q \in \cD'(\bbR)$ is real valued and one of the conditions $(i)$--$(iii)$ in Theorem \ref{t3.1}
is satisfied, then the form sum
\begin{equation}
\gQ_{T}(f,g) = \gQ_{T_0}(f,g) + q(\ol f g), \quad f, g \in \dom(\gQ_{T}) = H^1(\bbR),   \lb{3.17}
\end{equation}
defines a closed, densely defined, symmetric sesquilinear form $\gQ_{T}$ in $L^2(\bbR; dx)$,
bounded from below. The self-adjoint operator $T$ in $L^2(\bbR; dx)$, bounded from below, and
uniquely associated to the form $\gQ_{T}$ then can be described as follows,
\begin{align}
\begin{split}
& T f = \tau f, \quad \tau f = - (f' - q_2 f)' - q_2 (f' - q_2 f) + (q_1 - q_2^2)f,  \\
& f \in \dom(T) = \big\{g \in L^2(\bbR; dx) \,\big| \, g, (g' - q_2 g) \in AC_{\loc}(\bbR), \,
\tau g \in  L^2(\bbR; dx)\big\}.    \lb{3.18}
\end{split}
\end{align}
In particular, the differential expression $\tau$ formally corresponds to a Schr\"odinger operator
with distributional potential $q \in H_{\locunif}^{-1}(\bbR)$,
\begin{equation}
\tau = - (d^2/dx^2) + q(x), \quad q = q_1 + q_2', \quad q_j \in L^j_{\locunif}(\bbR; dx), \; j=1,2.
\lb{3.19}
\end{equation}
This is a consequence of the direct methods established in \cite{BS02},
\cite{HM01}--\cite{HM12}, \cite{KPST05}, \cite{SS99}, \cite{SS03}, \cite{We87},
and of the Weyl--Titchmarsh theory approach to Schr\"odinger operators with
distributional potentials developed in \cite{EGNT12a} (see also \cite{EGNT12}, \cite{EGNT12b},
and the detailed list of references therein). In particular, since $\tau$ is assumed to be bounded from
below, $\tau$ is in the limit point case at $\pm \infty$, rendering the maximally defined operator
$T$ in \eqref{3.18} to be self-adjoint (see also \cite{AKM10} and \cite{EGNT12a}). We will provide
further details on $\dom(T)$ in Remark \ref{r3.8}.
\end{remark}

Next, we turn to an elementary alternative approach to this circle of ideas in the real-valued context,
based on the concept of Miura transformations (cf.\ \cite{BGGSS87}, \cite{De78}, \cite{EGNT12},
\cite{Ge91}, \cite{Ge92}, \cite{GSS91}, \cite{GS95}, \cite{KPST05}, \cite{KT04},
\cite[Ch.\ 5]{Th92}, and the extensive literature cited therein) 
\begin{equation}
\begin{cases} L^2_{\loc}(\bbR; dx) \to H^{-1}_{\loc}(\bbR) \\
\phi \mapsto \phi^2 -  \phi^{\prime} \end{cases}     \lb{3.20}
\end{equation}
with associated self-adjoint Schr\"odinger operator $T_1 \geq 0$ in $L^2(\bbR; dx)$ given by
\begin{equation}
T_1 = A^* A,      \lb{3.21}
\end{equation}
with $A$ the closed operator defined in in $L^2(\bbR; dx)$ by
\begin{align}
\begin{split}
& A f = \alpha f, \quad \alpha f = f' + \phi f,  \\
& f \in \dom(A) = \big\{g \in L^2(\bbR; dx) \,\big| \, g \in AC_{\loc}(\bbR), \,
\alpha g \in  L^2(\bbR; dx)\big\},      \lb{3.22}
\end{split}
\end{align}
implying,
\begin{align}
\begin{split}
& A^* f = \alpha^+ f, \quad \alpha^+ f = - f' + \phi f,   \\
& f \in \dom(A^*) = \big\{g \in L^2(\bbR; dx) \,\big| \, g \in AC_{\loc}(\bbR), \,
\alpha^+ g \in  L^2(\bbR; dx)\big\}.    \lb{3.23}
\end{split}
\end{align}
Closedness of $A$ and the fact that $A^*$ is given by \eqref{3.23} was proved in \cite{KPST05}
(the extension to $\phi \in L^1_{\loc}(\bbR; dx)$, $\phi$ real-valued, was treated in \cite{EGNT12}).
In addition, it was proved in \cite{KPST05} that
\begin{equation}
C_0^{\infty}(\bbR) \, \text{ is an operator core for $A$ and $A^*$.}    \lb{3.24}
\end{equation}
Thus, $T_1$ acts as,
\begin{align}
\begin{split}
& T_1 f = \tau_1 f, \quad \tau_1 f = \alpha^+ \alpha f = - (f' + \phi f)' + \phi (f' + \phi f),   \\
& f \in \dom(T_1) = \big\{g \in L^2(\bbR; dx) \,\big|\, g, \alpha g \in AC_{\loc}(\bbR),
\tau_1 g \in L^2(\bbR; dx)\big\}.      \lb{3.25}
\end{split}
\end{align}
In particular, $\tau_1$ is formally of the type,
\begin{equation}
\tau_1 = - (d^2/dx^2) + V_1(x), \quad
V_1 = \phi^2 - \phi', \quad \phi \in L^2_{\loc}(\bbR; dx),     \lb{3.26}
\end{equation}
displaying the Riccati equation connection between $V_1$ and $\phi$ in connection with Miura's
transformation \eqref{3.20}.

\begin{theorem} \lb{t3.4} $($\cite{KPST05}.$)$
Assume that $q \in H^{-1}_{\loc}(\bbR)$ is real-valued.
Then the following conditions $(i)$--$(iii)$ are equivalent: \\[1mm]
$(i)$ $q = \phi^2 - \phi'$ for some real-valued $\phi \in L^2_{\loc}(\bbR; dx)$. \\[1mm]
$(ii)$ $(- d^2/dx^2) + q \geq 0$ in the sense of distributions, that is,
\begin{equation}
(f',f')_{L^2(\bbR; dx)} + q({\ol f} f)
= {}_{H^{-1}(\bbR)}\langle (- f'' + q f), f\rangle_{H^1(\bbR)} \geq 0 \, \text{ for all } \,
f \in C_0^{\infty} (\bbR).    \lb{3.48}
\end{equation}
$(iii)$ $[(- d^2/dx^2) + q] \psi = 0$ has a positive solution
$0 < \psi \in H^1_{\loc}(\bbR)$.
\end{theorem}

\begin{theorem} \lb{t3.5} $($\cite{KPST05}.$)$
Assume that $q \in H^{s-1}(\bbR)$, $s\geq 0$, is real-valued.
Then the following conditions $(i)$ and $(ii)$ are equivalent: \\[1mm]
$(i)$ $q = \phi^2 - \phi'$ for some real-valued $\phi \in H^s(\bbR)$. \\[1mm]
$(ii)$ $(- d^2/dx^2) + q \geq 0$ in the sense of distributions $($cf.\ \eqref{3.48}$)$
and $q = q_1 + q_2^{\prime}$ for some $q_j \in L^j(\bbR; dx)$, $j=1,2$.
\end{theorem}

The following appears to be a new result:

\begin{theorem} \lb{t3.6}
Assume that $q \in H^{-1}_{\locunif}(\bbR)$ is real-valued.
Then the following conditions $(i)$--$(iii)$ are equivalent: \\[1mm]
$(i)$ $q = \phi^2 - \phi'$ for some real-valued
$\phi \in L^2_{\locunif}(\bbR; dx)$. \\[1mm]
$(ii)$ $(- d^2/dx^2) + q \geq 0$ in the sense of distributions $($cf.\ \eqref{3.48}$)$. \\[1mm]
$(iii)$ $[(- d^2/dx^2) + q] \psi = 0$ has a positive solution
$0 < \psi \in H^1_{\loc}(\bbR)$.
\end{theorem}
\begin{proof}
We will show that $(ii) \Longrightarrow (iii) \Longrightarrow (i) \Longrightarrow (ii)$.

Given item $(ii)$, that is, $q \in H^{-1}_{\locunif}(\bbR)$ is real-valued and
$(- d^2/dx^2) + q \geq 0$, one concludes the existence of $0 < \psi_0 \in H^1_{\loc}(\bbR)$
such that $- \psi_0'' + q \psi_0 =0$ by Theorem \ref{t3.4}\,$(iii)$. Thus, item $(iii)$ follows. \\
Introducing
\begin{equation}
\phi_0 = - \psi_0'/\psi_0,
\end{equation}
one infers that
\begin{equation}
\phi_0 \in L^2_{\loc}(\bbR; dx) \, \text{ is real-valued and } \, q = \phi_0^2 - \phi_0'.    \lb{3.30}
\end{equation}
Next, introducing $A_0^{}$ and $A_0^*$ as in \eqref{3.22} and \eqref{3.23}, with $\alpha$
replaced by $\alpha_0^{} = (d/dx) + \phi_0$ (and analogously for $\alpha^+$), we now introduce
the sesquilinear form $\dot \gQ_{T_1}$ and its closure, $\gQ_{T_1}$, by
\begin{align}
\begin{split}
& \dot \gQ_{T_1} (f,g) = (A_0^{} f, A_0^{} g)_{L^2(\bbR; dx)}, \quad
f, g \in \dom\big(\dot \gQ_{T_1}\big) = C_0^{\infty} (\bbR),    \lb{3.31} \\
& \gQ_{T_1} (f,g) = (A_0^{} f, A_0^{} g)_{L^2(\bbR; dx)}, \quad f, g \in \dom(\gQ_{T_1})
= \dom(A_0^{}),
\end{split}
\end{align}
with $0 \leq T_1 = A_0^* A_0^{}$ the uniquely associated self-adjoint operator.

Since by hypothesis $q \in H^{-1}_{\locunif}(\bbR)$, it is known (cf.\ \cite{HM01}) that $q$ can be
written as
\begin{equation}
q = q_1 + q_2' \, \text{ for some } \, q_j \in L_{\locunif}^j (\bbR; dx), \; j = 1,2,
\end{equation}
and hence, we also introduce the sesquilinear form $\dot{\hatt{\gQ}}$ and its closure,
$\hatt{\gQ}$ (cf.\ \cite{HM01} for details),
\begin{align}
\begin{split}
\dot {\hatt{\gQ}}(f,g) &= (f',g')_{L^2(\bbR; dx)} - (f', q_2 g)_{L^2(\bbR; dx)}
- (q_2 f, g')_{L^2(\bbR; dx)}     \\
& \quad + \big(|q_1|^{1/2} f, \sgn(q_1) |q_1|^{1/2} g\big)_{L^2(\bbR; dx)},
\quad  f, g \in \dom\big(\dot {\hatt{\gQ}}\big) = C_0^{\infty}(\bbR),
\end{split} \\
\begin{split}
\hatt{\gQ}(f,g) &= (f',g')_{L^2(\bbR; dx)} - (f', q_2 g)_{L^2(\bbR; dx)}
- (q_2 f, g')_{L^2(\bbR; dx)}     \\
& \quad + \big(|q_1|^{1/2} f, \sgn(q_1) |q_1|^{1/2} g\big)_{L^2(\bbR; dx)},
\quad  f, g \in \dom\big(\hatt{\gQ}\big) = H^1(\bbR).
\end{split}
\end{align}
Since
\begin{equation}
\gQ_{T_1} (f,g) = \hatt{\gQ}(f,g) = (f',g')_{L^2(\bbR; dx)} + q(\ol f g), \quad
f, g \in C_0^{\infty} (\bbR),
\end{equation}
and $C_0^{\infty}(\bbR)$ is a form core for $\gQ_{T_1}$ (cf.\ \eqref{3.24}) and $\hatt{\gQ}$,
one concludes that $\gQ_{T_1} = \hatt{\gQ}$ and hence
\begin{equation}
\dom(\gQ_{T_1}) = \dom(A_0) = \dom\big(\hatt{\gQ}\big) = H^1(\bbR).    \lb{3.36}
\end{equation}
A comparison of \eqref{3.22} (with $\alpha$ replaced by $\alpha_0$) and \eqref{3.36} implies that
$\phi_0 g \in L^2(\bbR; dx)$ for $g \in \dom(A_0^{}) = H^1(\bbR)$, and hence,
\begin{equation}
\dom(\phi_0) \supseteq H^1(\bbR).      \lb{3.37}
\end{equation}
An application of Theorem \ref{t2.1}\,$(i), (ii)$ then finally yields
\begin{equation}
\phi_0 \in L^2_{\locunif} (\bbR; dx),     \lb{3.38}
\end{equation}
which together with \eqref{3.30} implies item $(i)$.

Finally, given $\phi \in L^2_{\locunif}(\bbR; dx)$, $\phi$ real-valued, such that $q = \phi^2 - \phi'$,
one computes, with $\alpha = (d/dx) + \phi$,
\begin{align}
\begin{split}
0 \leq \|\alpha f\|_{L^2(\bbR; dx)}^2 = \|f'\|_{L^2(\bbR; dx)}^2 + q\big(|f|^2\big)
= {}_{H^{-1}(\bbR)}\langle (- f'' + q f), f\rangle_{H^1(\bbR)},& \\
f \in C_0^{\infty}(\bbR),&
\end{split}
\end{align}
and hence item $(i)$ implies item $(ii)$.
\end{proof}

Thus, Theorem \ref{t3.6} further illustrates the results by Bak and
Shkalikov \cite{BS02} and Maz'ya and Verbitsky \cite{MV05} (specialized to the 
one-dimensional situation) recorded in Theorem \ref{t3.1} in the particular case 
where $q$ is real-valued.

In connection with Theorem \ref{t3.6}\,$(i)$, we also recall the following useful result:

\begin{lemma} $($\cite{HM12}.$)$ \lb{l3.7}
Assume that $q \in H^{-1}_{\locunif}(\bbR)$ is real-valued and of the form
$q = \phi^2 - \phi'$ for some real-valued $\phi \in L^2_{\loc}(\bbR; dx)$. Then, actually,
\begin{equation}
\phi \in L^2_{\locunif}(\bbR; dx).
\end{equation}
\end{lemma}

\begin{remark} \lb{r3.8}
Combining \eqref{3.17}--\eqref{3.19}, \eqref{3.25}, \eqref{3.26}, \eqref{3.31}, and \eqref{3.37}
(identifying $\phi$ and $\phi_0$ as well as $T$ and $T_1$) then yields the following apparent improvement over the domain characterizations \eqref{3.18}, \eqref{3.25},
\begin{align}
& T_1 f = \tau_1 f, \quad \tau f = - (f' + \phi f)' + \phi (f' + \phi f),    \no \\
& f \in \dom(T_1) = \big\{g \in L^2(\bbR; dx) \,\big|\, g, \alpha g \in AC_{\loc}(\bbR),
g', \phi g \in L^2(\bbR; dx),       \lb{3.39} \\
& \hspace*{8.1cm} \tau_1 g \in L^2(\bbR; dx)\big\},      \no
\end{align}
with \eqref{3.26} staying in place. In fact, \eqref{3.25} and \eqref{3.39} are, of course,  equivalent; the former represents a minimal characterization of $\dom(T_1)$.
\end{remark}

\begin{remark} \lb{r3.9}
Given $q = \phi^2 - \phi'$, $\phi \in L^2_{\locunif}(\bbR; dx)$ as in Theorems \ref{t3.4} -- \ref{t3.6},
the question of uniqueness of $\phi$ for prescribed $q \in H^{-1}_{\loc}(\bbR)$ arises naturally.
This has been settled in \cite{KPST05} and so we briefly summarize some pertinent facts. Since
$\phi = - \psi'/\psi$ for some $0 < \psi \in H^1_{\loc}(\bbR)$, uniqueness of $\phi$ is equivalent to
uniqueness of $\psi > 0$ satisfying $[(- d^2/dx^2) + q]\psi = 0$. Thus, suppose
$0 < \psi_0 \in H^1_{\loc}(\bbR)$ is a solution of $[(- d^2/dx^2) + q]\psi = 0$. Then, the general,
real-valued solution of $[(- d^2/dx^2) + q]\psi = 0$ is of the type
\begin{equation}
\psi(x) = C_1 \psi_0(x) + C_2 \psi_0(x) \int_{0}^x dx' \, \psi_0(x')^{-2}, \quad x \in \bbR, \;
C_j \in \bbR, \; j=1,2.
\end{equation}
Next, introducing
\begin{equation}
c_{\pm} = \pm \lim_{x \to \pm \infty} \int_0^x dx' \, \psi_0(x')^{-2} \in (0, + \infty],
\end{equation}
and defining $c_{\pm}^{-1} = 0$ if $c_{\pm} = + \infty$, all positive solutions
$0 < \psi$ on $\bbR$ of $[(- d^2/dx^2) + q]\psi = 0$ are given by
\begin{equation}
\psi(x) = \psi_0(x) \bigg[1 + c \int_{0}^x dx' \, \psi_0(x')^{-2} \bigg], \quad
c \in \big[- c_+^{-1}, c_-^{-1}\big].    \lb{3.44}
\end{equation}
Consequently,
\begin{align}
\begin{split}
& 0 < \psi_0 \in H^1_{\loc}(\bbR) \, \text{ is the unique solution of } \, [(- d^2/dx^2) + q]\psi = 0 \\
&\quad \text{if and only if } \, \pm \lim_{x \to \pm \infty} \int_{0}^x dx' \, \psi_0(x')^{-2} = \infty.
\end{split}
\end{align}
On the other hand, if at least one of $ \pm \lim_{x \pm \infty} \int_{0}^x dx' \, \psi_0(x')^{-2} < \infty$,
$[(- d^2/dx^2) + q]\psi = 0$ has a one (real) parameter family of positive solutions on $\bbR$
lying in $H^1_{\loc}(\bbR)$ given by \eqref{3.44}. Without going into further details, we note
that Weyl--Titchmarsh solutions $\psi_{\pm}(\lambda,\cdot)$ corresponding to $T$ in
\eqref{3.18} for energies  $\lambda < \inf(\sigma(T))$, are actually constant multiples of Hartman's principal solutions $T \hatt \psi_{\pm} (\lambda, \cdot) = \lambda \hatt \psi_{\pm}(\lambda, \cdot)$,
that is, those that satisfy
$ \pm \int^{\pm \infty} dx' \, \big[\hatt \psi_{\pm}(\lambda,x')\big]^{-2} = \infty$.
\end{remark}

\begin{theorem} \lb{t3.10} Assume that $q \in H^{-1}_{\loc} (\bbR)$ is real-valued and suppose
in addition that $(- d^2/dx^2) + q \geq 0$ in the sense of distributions $($cf.\ \eqref{3.48}$)$.
 Then the following conditions $(i)$--$(iv)$ are equivalent: \\[1mm]
$(i)$ $q$ is form compact with respect to $T_0$, that is, the map
\begin{equation}
q \colon H^1(\bbR) \to H^{-1}(\bbR) \, \text{ is compact.}    \lb{3.45a}
\end{equation}
$(ii)$ $q$ is of the form $q = \phi^2 - \phi'$, where $\phi \in L^2_{\locunif}(\bbR; dx)$ is
real-valued and
\begin{equation}
\lim_{|a|\to\infty} \bigg(\int_a^{a+1} dx \, \phi(x)^2\bigg) = 0.  \lb{3.45b}
\end{equation}
$(iii)$ The operator of multiplication by $\phi$ is $T_0^{1/2}$-compact. \\[1mm]
$(iv)$ The operator of multiplication by $\phi$ is $P_m\big(T_0^{1/2}\big)$-compact, where
$P_m$ is a polynomial of degree $m\in\bbN$.
\end{theorem}
\begin{proof}
By Theorem \ref{t3.4}, $(- d^2/dx^2) + q \geq 0$ in the sense of distributions implies that
$q$ is of the form $q = \phi^2 - \phi'$ for some real-valued $\phi \in L^2_{\loc} (\bbR)$. By
Lemma \ref{l3.7}, one actually concludes that $\phi \in L^2_{\locunif} (\bbR)$. The equivalence
of items $(i)$ and $(ii)$ then follows from Theorem \ref{t3.2}
since upon identifying $q_1 = \phi^2$, $q_2 = \phi$, the two limiting relations in \eqref{3.16} are
equivalent to \eqref{3.45b}. Equivalence of condition \eqref{3.45b} and item $(iii)$ is guaranteed
by Theorem \ref{t2.2}.
\end{proof}

At this point it is worth recalling a few additional details of the supersymmetric formalism started
in \eqref{3.20}--\eqref{3.26}, whose abstract roots can be found in Appendix \ref{sB}:
Assuming $\phi \in L^2_{\locunif}(\bbR; dx)$ to be real-valued
(we note, however, that this supersymmetric formalism extends to the far more general situation
where $\phi \in L^1_{\loc}(\bbR; dx)$ is real-valued, in fact, it extends to the situation where $\phi$
is matrix-valued, see \cite{EGNT12} for a detailed treatment of these matters), one has
\begin{align}
& A = (d/dx) + \phi, \quad A^* = - (d/dx) + \phi,   \quad
\dom(A) = \dom(A^*) = H^1(\bbR),    \lb{3.46} \\
& T_1 = A^* A = - (d^2/dx^2) + V_1, \quad V_1 = \phi^2 - \phi',     \lb{3.47} \\
& T_2 = A A^* = - (d^2/dx^2) + V_2, \quad V_2 = \phi^2 + \phi',     \lb{3.70} \\
& D = \begin{pmatrix} 0 & A^* \\ A & 0 \end{pmatrix}
\, \text{ in $L^2(\bbR; dx) \oplus L^2(\bbR; dx)$,}     \lb{3.49} \\
& D^2 = \begin{pmatrix} A^* A & 0 \\ 0 & A A^* \end{pmatrix}
= T_1 \oplus T_2 \, \text{ in $L^2(\bbR; dx) \oplus L^2(\bbR; dx)$.}    \lb{3.50}
\end{align}
As a consequence, one can show (cf.\ \cite{EGNT12}) the Weyl--Titchmarsh solutions, $\psi_{\pm}$,
for $D, T_1, T_2$ satisfy
\begin{align}
& \psi_{D, 1, \pm} (\zeta,x) = \psi_{T_1, \pm} (z,x), \quad z = \zeta^2, \; \zeta \in \bbC\backslash\bbR, \\
& \psi_{T_2, \pm} (z,x)
= c_1(z) (A \psi_{T_1, \pm}) (z,x),
\end{align}
with $c_1(z)$ a normalization constant. Similarly, after interchanging the role of $T_1$ and $T_2$,
\begin{align}
& \psi_{D, 2, \pm} (\zeta,x) = \psi_{T_2, \pm} (z,x), \quad z = \zeta^2, \; \zeta \in \bbC\backslash\bbR, \\
& \psi_{T_1, \pm} (z,x)
= c_2(z) (A^* \psi_{T_2, \pm}) (z,x),
\end{align}
again with $c_2(z)$ a normalization constant. Here,
\begin{equation}
\Psi_{D,\pm} (\zeta,x) = \begin{pmatrix} \psi_{D, 1, \pm} (\zeta,x) \\
\psi_{D, 2, \pm} (\zeta,x) \end{pmatrix}
\end{equation}
are the Weyl--Titchmarsh solutions for
$D = \left(\begin{smallmatrix} 0 & A^* \\ A & 0 \end{smallmatrix}\right)$.

The (generalized, or renormalized) Weyl--Titchmarsh $m$-functions for $D, T_1, T_2$ satisfy:
\begin{equation}
m_{D,\pm} (\zeta,x_0) = \f{1}{\zeta} \hatt m_{T_1,\pm} (z,x_0)
= \f{-\zeta}{\hatt m_{T_2,\pm} (z,x_0)},
\end{equation}
where $x_0$ is a fixed reference point (typically, $x_0 = 0$), and
\begin{align}
& \hatt m_{T_1,\pm} (z,x_0)
= \f{\psi_{T_1, \pm}^{[1,1]} (z,x_0)}{\psi_{T_1, \pm} (z,x_0)}
= \f{(A \psi_{T_1, \pm}) (z,x_0)}{\psi_{T_1, \pm} (z,x_0)},     \lb{3.57} \\
& \hatt m_{T_2,\pm} (z,x_0)
= \f{\psi_{T_2, \pm}^{[1,2]} (z,x_0)}{\psi_{T_2, \pm} (z,x_0)}
= \f{(- A^* \psi_{T_2, \pm}) (z,x_0)}{\psi_{T_2, \pm} (z,x_0)}.     \lb{3.58}
\end{align}
Here, $y^{[1,1]} = A y = [y' + \phi y]$ is the quasi-derivative corresponding to $T_1$ and
$y^{[1,2]} = - A^* y = [y' - \phi y]$ is the quasi-derivative corresponding to $T_2$.

Thus, spectral properties of $D$ instantly translate into spectral properties of $T_j$, $j=1,2$,
and vice versa (the latter with the exception of the zero spectral parameter). In particular,
$\phi \in L^2_{\locunif}(\bbR; dx) \subset L^2_{\loc}(\bbR; dx)$ in $D$ is entirely ``standard'' (in
fact, even $\phi \in L^1_{\loc}(\bbR; dx)$ in $D$ is entirely standard, see, e.g., \cite{CG02} and the
extensive literature cited therein), while the potentials $V_j = (-1)^j \phi' + \phi^2$, $j=1,2$, involve 
the distributional coefficient $\phi' \in H^{-1}_{\locunif} (\bbR)$. (We also note that while in this paper the Dirac operator $D$ only involves the $L^2_{\loc}(\bbR; dx)$-coefficient $\phi$, Dirac-type operators with distributional potentials have been studied in the literature, see, for instance \cite[App.\ J]{AGHKH05}  
and \cite{CMP13}.) In particular, spectral results for the ``standard'' one-dimensional
Dirac-type operator $D$ imply corresponding spectral results for Schr\"odinger operators bounded
from below, with (real-valued) distributional potentials. Some applications of this spectral
correspondence between $D$ and $T_j$, $j=1,2$, to inverse spectral theory, local Borg--Marchenko uniqueness results, etc., were treated in \cite{EGNT12}. In Section \ref{s4} we will apply this
spectral correspondence to derive some Floquet theoretic results in connection with the Schr\"odinger operators $T_j$ and hence for the distributional potentials
$[\phi^2 + (-1)^j \phi'] \in H^{-1}_{\locunif}(\bbR)$, $j=1,2$.

\begin{remark} \lb{r3.12}
For simplicity we restricted ourselves to the special case $p=1$ in Theorems \ref{t3.4}--\ref{t3.6}
and Remarks \ref{r3.8} and \ref{r3.9}. However, assuming
\begin{equation}
0 < p, p^{-1} \in L^{\infty}(\bbR; dx), \quad 0 < r, r^{-1} \in L^{\infty}(\bbR; dx),
\end{equation}
the observations thus far in this section extend to the case where
\begin{align}
\begin{split}
\tau_1 f = \alpha^+ \alpha f &= - f'' + \big[\phi^2 - \phi' \big] f    \\
&= - (f' + \phi f)' + \phi(f' + \phi f)
\end{split}
\end{align}
is replaced by
\begin{align}
\begin{split}
\tau_1 f = \beta^+ \beta f &= r^{-1} \Big[- (p f')' + \big[p \phi^2 - (p\phi)' \big]f\Big]     \\
&= r^{-1} \Big[- [p(f' + \phi f)]' + \phi [p (f' + \phi f)]\Big],
\end{split}
\end{align}
where
\begin{align}
\begin{split}
\beta f &= (pr)^{-1/2} [p(f' + \phi f)],      \\
\beta^+ f &= - (pr)^{-1} \Big\{p \Big[\big[(pr)^{1/2} f\big]' - \phi \big[(pr)^{1/2} f\big]\Big]\Big\}.
\end{split}
\end{align}
\end{remark}

\begin{remark} \lb{r3.13} We only dwelled on
\begin{equation}
\dom\big(|T|^{1/2}\big) = H^1(\bbR)
\end{equation}
to derive a number of {\it if and only if} results. For practitioners in this field, the sufficient
conditions on $q, r$ in terms of the $L^j_{\locunif}(\bbR; dx)$, $j=1,2$, and boundedness
conditions on $0 < p, p^{-1}$, yielding form boundedness (i.e., self-adjointness) results, relative compactness, and trace class results, all work as long as one ensures
\begin{equation}
\dom\big(|T|^{1/2}\big) \subseteq H^1(\bbR).
\end{equation}
This permits larger classes of coefficients $p, q, r$ for which one can prove these types of self-adjointness and spectral results.
\end{remark}

\medskip

Before returning to our principal object, the Birman--Schwinger-type operator $\ul{T^{-1/2} r T^{-1/2}}$,
we briefly examine the well-known example of point interactions:

\begin{example} \lb{e3.14} $($Delta distributions.$)$
\begin{equation}
q_1(x) = 0, \quad q_2 (x) = \begin{cases} 1, & x > x_0, \\
0, & x < x_0, \end{cases} \quad \text{then} \quad
q = q_2' = \delta_{x_0}, \quad x_0 \in \bbR.
\end{equation}
Introducing the operator
\begin{align}
& \, A_{\alpha, x_0} = \f{d}{dx} - \f{\alpha}{2} \begin{cases} \;\;\, 1, & x>x_0, \\ -1, & x<x_0,
\end{cases} \quad  \text{ that is, } \,
\phi(x) = \f{\alpha}{2} \sgn(x-x_0), \quad \alpha, x_0 \in \bbR, \;   \no \\
& \dom(A_{\alpha, x_0}) = H^1(\bbR),
\end{align}
in $L^2(\bbR; dx)$, one infers that 
\begin{equation}
A_{\alpha, x_0}^* A_{\alpha, x_0} = - \Delta_{\alpha, x_0} + (\alpha^2/4) I. 
\end{equation}
Here $- \Delta_{\alpha, x_0} = - d^2/dx^2 + \alpha \, \delta_{x_0}$ in $L^2(\bbR; dx)$ represents 
the self-adjoint realization of the one-dimensional point interaction $($cf.\ 
\cite[Ch.\ I.3]{AGHKH05}$)$, that is, 
the Schr\"odinger operator with a delta function potential of strength $($coupling 
constant\,$)$ $\alpha$ centered at $x_0 \in \bbR$.
\end{example}
This extends to sums of delta distributions supported on a discrete set
(Kronig--Penney model, etc.).

Next we apply this distributional approach to the Birman--Schwinger-type operator
$\ul{T^{-1/2} r T^{-1/2}}$. We outline the basic ideas in the following three steps: \\[1mm]
{\bf Step 1.} Assume $p, p^{-1} \in L^{\infty}(\bbR; dx)$, $p>0$ a.e.\ on $\bbR$. \\[1mm]
{\bf Step 2.} Suppose $q = q_1 + q_2'$, where $q_j \in L^j_{\locunif}(\bbR; dx)$, $j=1,2$,
are real-valued. This uniquely defines a self-adjoint operator $T$ in $L^2(\bbR; dx)$, bounded
from below, $T \geq c I$ for some $c \in \bbR$, as the form sum $T = - (d/dx) p (d/dx) + q$
of $-(d/dx) p (d/dx)$ and the distribution $q = q_1 + q_2' \in \cD^{\prime}(\bbR)$. Then
\begin{equation}
\dom\big(|T|^{1/2}\big) = H^1(\bbR).
\end{equation}

If in addition, $\lim_{|a|\to\infty} \Big(\int_a^{a+1} dx \, |q_1(x) - c_1|\Big) = 0$
for some constant $c_1 \in \bbR$, and \;
$\lim_{|a|\to\infty} \Big(\int_a^{a+1} dx \, |q_2(x)|^2\Big) = 0$, one again obtains results on essential
spectra. \\[1mm]
{\bf Step 3.} Suppose without loss of generality, that $T \geq c I$, $c > 0$, and introduce
$r = r_1 + r_2'$, $r_j \in L^j_{\locunif}(\bbR; dx)$ real-valued, $j=1,2$.
This uniquely defines a bounded self-adjoint operator
$\ul{T^{-1/2} r T^{-1/2}}$ in $L^2(\bbR; dx)$ as described next: First write
\begin{align}
& \ul{T^{-1/2} r T^{-1/2}}    \\
& \quad = \big[(T_0 + I)^{1/2} T^{-1/2}\big]^*
\big[(T_0 + I)^{-1/2} r (T_0 + I)^{-1/2}\big]
\big[(T_0 + I)^{1/2} T^{-1/2}\big].   \no
\end{align}
Next, one interprets $(T_0 + I)^{-1/2} r (T_0 + I)^{-1/2}$ as follows: Employing $T_0$ and its
extension, $\wti T_0$, to the entire Sobolev scale $H^s(\bbR)$ in \eqref{3.11}--\eqref{3.14c},
in particular, we will employ the mapping properties,  $\big(\wti T_0 + I\big)^{-1/2}: H^s(\bbR) \to H^{s+1}(\bbR)$,
$s \in \bbR$. Thus, using
\begin{equation}
\big[(T_0 + I)^{1/2} T^{-1/2}\big], \big[(T_0 + I)^{1/2} T^{-1/2}\big]^*
\in \cB\big(L^2(\bbR; dx)\big),
\end{equation}
and
\begin{equation}
\Big[\underbrace{\big(\wti T_0 + I\big)^{-1/2}}_{\in \cB(H^{-1}(\bbR), L^2(\bbR; dx))} \;\;
\underbrace{r}_{\in \cB(H^1(\bbR), H^{-1}(\bbR))} \;\;
\underbrace{\big(\wti T_0 + I\big)^{-1/2}}_{\in \cB(L^2(\bbR; dx), H^1(\bbR))}\Big] \in
\cB\big(L^2(\bbR; dx)\big),
\end{equation}
finally yields
\begin{align}
& \ul{T^{-1/2} r T^{-1/2}}  \no \\
& \quad = \big[(T_0 + I)^{1/2} T^{-1/2}\big]^*
\Big[\big(\wti T_0 + I\big)^{-1/2} r \big(\wti T_0 + I\big)^{-1/2}\Big]
\big[(T_0 + I)^{1/2} T^{-1/2}\big]   \no \\
& \qquad \in \cB\big(L^2(\bbR; dx)\big).      \lb{3.74}
\end{align}

Equivalently, using the analog of \eqref{C.24a} applied to $T$

Hence, our reformulated left-definite generalized eigenvalue problem becomes again a standard
self-adjoint spectral problem in $L^2(\bbR; dx)$,
\begin{equation}
\ul{T^{-1/2} r T^{-1/2}} \, \chi = \f{1}{z} \chi, \quad z \in \bbC\backslash\{0\},    \lb{3.75}
\end{equation}
associated with the bounded, self-adjoint operator $\ul{T^{-1/2} r T^{-1/2}}$ in $L^2(\bbR; dx)$,
yet this time we permit distributional coefficients satisfying
\begin{align}
& p, p^{-1} \in L^{\infty}(\bbR; dx), \; p>0 \text{ a.e.\ on } \bbR,     \\
& q = q_1 + q_2', \; q_j \in L^j_{\locunif}(\bbR; dx) \text{ real-valued,} \; j=1,2,   \\
& r = r_1 + r_2', \; r_j \in L^j_{\locunif}(\bbR; dx) \text{ real-valued,} \; j=1,2,
\end{align}
with $T$ defined as the self-adjoint, lower-semibounded operator
uniquely associated with the lower-bounded, closed sesquilinear form $\gQ_T$ in $L^2(\bbR; dx)$
given by
\begin{align}
\gQ_T(f,g) &= \big(p^{1/2} f', p^{1/2} g'\big)_{L^2(\bbR; dx)} + q(\ol f g)    \\
&= \big(p^{1/2} f', p^{1/2} g'\big)_{L^2(\bbR; dx)} - (f', q_2 g)_{L^2(\bbR; dx)}
- (q_2 f, g')_{L^2(\bbR; dx)}     \no \\
& \quad + \big(|q_1|^{1/2} f, \sgn(q_1) |q_1|^{1/2} g\big)_{L^2(\bbR; dx)},     \\
&= \big(p^{-1/2}(pf' - q_2 f), p^{-1/2}(pg' - q_2 g)\big)_{L^2(\bbR; dx)}      \\
& \quad + \big(|q_1|^{1/2}f, \sgn(q_1) |q_1|^{1/2}g\big)_{L^2(\bbR; dx)}
-  \big(p^{-1/2} q_2 f, p^{-1/2} q_2 g\big)_{L^2(\bbR; dx)},     \no \\
& \hspace*{6.65cm} f, g \in \dom(\gQ_T) = H^1(\bbR).      \no
\end{align}
In particular, $T$ corresponds to the differential expression $\tau = - (d/dx) p (d/dx) + q(x)$,
$x \in \bbR$, and hence is explicitly given by
\begin{align}
& T f = \tau f, \quad \tau f = - (pf' - q_2 f)' - p^{-1} q_2 (p f' - q_2 f) + \big(q_1 - p^{-1} q_2^2\big)f,  \no \\
& f \in \dom(T) = \big\{g \in L^2(\bbR; dx) \,\big| \, g, (pg' - q_2 g) \in AC_{\loc}(\bbR), \,
\tau g \in  L^2(\bbR; dx)\big\}.    \no \\
& \hspace*{1.88cm} = \big\{g \in L^2(\bbR; dx) \,\big| \, g, (pg' - q_2 g) \in AC_{\loc}(\bbR),
\, \tau g \in L^2(\bbR; dx)   \no \\
& \hspace*{4.85cm} (pg' - q_2 g) \in L^2(\bbR; dx)\big\}.    \lb{3.81}
\end{align}
Without loss of generality we assume $T \geq c I$ for some $c>0$.

\section{The Case of Periodic Coefficients}   \label{s4}

In this section we apply some of the results collected in Sections \ref{s2} and \ref{s3} to the
special, yet important, case where all coefficients are periodic with a fixed period. For simplicity,
we will choose $p=1$ throughout, but we emphasize that including the nonconstant, periodic
coefficient $p$ can be done in a standard manner as discussed in Remark \ref{r3.12}. It is
not our aim to present a thorough treatment of Floquet theory, rather, we intend to illustrate
some of the scope underlying the approach developed in this paper.

One recalls that $q\in H^{-1}_{\loc}(\bbR)$ is called {\it periodic with period $\omega > 0$} if 
\begin{equation} 
\dpl q, f(\cdot-\omega) \dpr q = \dpl q, f \dpr, \quad f \in H^1(\bbR). 
\end{equation} 
By \eqref{3.29}, if $q\in H^{-1}_{\loc}(\bbR)$ is periodic, it can be written as $q=q_1+q_2'$, 
where $q_1$ is a constant and 
$q_2\in L^2_{\locunif}(\bbR; dx)$ is periodic with period $\omega$. The analogous statement 
applies, of course, to the coefficient $r$ in the differential equation \eqref{1.7}. Introducing the abbreviations $Q = q - zr$, $Q_1 = q_1 - zr_1$, and $Q_2 = q_2 - zr_2$ and the 
quasi-derivative $y^{[1]} = y'-Q_2y$ we may now write \eqref{1.7} as
\begin{equation} \label{r4.1}
  \tau y=-(y^{[1]})^\prime-Q_2y^{[1]}+(Q_1-Q_2^2)y=0, 
\end{equation}
or, equivalently, as the first-order system
\begin{equation}
\begin{pmatrix} y\\ y^{[1]}\end{pmatrix}'=
\begin{pmatrix} Q_2&1\\ Q_1-Q_2^2&-Q_2\end{pmatrix}
\begin{pmatrix} y\\ y^{[1]}\end{pmatrix}.
\end{equation}
Existence and uniqueness for the corresponding initial value problem as well as the constancy 
of the modified Wronskian,  
\begin{equation}
W(f,g)(x) = f(x)g^{[1]}(x)-f^{[1]}(x)g(x), 
\end{equation}
were established in \cite{EGNT12a}. As a consequence the monodromy map 
\begin{equation} 
M(z): y\mapsto y(\cdot+\omega) 
\end{equation} 
maps the two-dimensional space of solutions of equation \eqref{r4.1} onto itself and has 
determinant $1$ (as usual this is seen most easily by introducing a standard basis 
$u_1$, $u_2$ defined by the initial values $u_1(c)=u_2^{[1]}(c)=1$ and 
$u_1^{[1]}(c)=u_2(c)=0$). The trace of $M(z)$, given by 
$u_1(c+\omega)+u_2^{[1]}(c+\omega)$, is real which implies that the eigenvalues 
$\rho(z)$ and $1/\rho(z)$ of $M(z)$ (the Floquet multipliers) are either both real, or else, 
are complex conjugates of each other, in which case they both lie on the unit circle. The proof 
of Theorem\ 2.7 in \cite{EGNT12a} may also be adapted to show that, for each fixed point $x$, 
the functions  $u_1(x)$, $u_2(x)$, $u_1^{[1]}(x)$, and $u_2^{[1]}(x)$ are entire functions of growth 
order $1/2$ with respect to $z$. In particular, $\tr_{\bbC^2} (M(\cdot))$ is an entire function of growth 
order $1/2$.

We start by focusing on the operator $T$ as discussed in \eqref{3.17}--\eqref{3.19}.

Throughout this section we make the following assumptions: 

\begin{hypothesis} \lb{h4.1}
Assume that $q \in H^{-1}_{\loc}(\bbR)$ is real-valued and periodic with period
$\omega > 0$ $($and hence, actually, $q \in H^{-1}_{\locunif}(\bbR)$$)$. Define
$T$ in $L^2(\bbR; dx)$ according to \eqref{3.17}--\eqref{3.19} and suppose that $T \geq 0$.
\end{hypothesis}

\begin{lemma} \lb{l4.2}
Assume Hypothesis \ref{h4.1}. Then there exists $\phi_0 \in L^2_{\locunif}(\bbR;dx)$,
real-valued and periodic of period $\omega > 0$, such that $q = \phi^2_0 - \phi'_0$.
\end{lemma} 
\begin{proof}
It suffices to note that (as in the standard case where $q \in L^1_{\loc}(\bbR)$ is real-valued
and periodic with period $\omega > 0$) the Weyl--Titchmarsh solutions $\psi_{T,\pm}(z,\cdot \,)$
satisfy
\begin{equation}
\psi_{T,\pm}(z,x) > 0, \quad z < 0, \; x \in \bbR,
\end{equation}
which extends by continuity to $z=0$, that is,
\begin{equation}
\psi_{T,\pm}(0,x) > 0, \quad x \in \bbR,
\end{equation}
although, $\psi_{T,\pm}(0,\cdot)$ may no longer lie in $L^2$ near $\pm\infty$ and hence cease
to be a Weyl--Titchmarsh solution. (By oscillation theory, cf.\ \cite{EGNT12a}, a zero of
$\psi_{T,\pm}(0,\cdot)$ would contradict $T \geq 0$.)
Using the Floquet property of $\psi_{T,\pm}(z,\cdot \,)$, $\phi_{\pm}$ defined by
\begin{equation}
\phi_{\pm}(x) = \psi_{T,\pm}'(0,x)/\psi_{T,\pm}(0,x), \quad x \in \bbR,
\end{equation}
satisfies
\begin{equation}
\phi_{\pm} \in L^2_{\loc}(\bbR), \, \text{ $\phi_{\pm}(\cdot)$ is periodic with period
$\omega > 0$,}
\end{equation}
in particular,
\begin{equation}
\phi_{\pm} \in L^2_{\locunif}(\bbR) \, \text{ and $q = \phi^2_{\pm} - \phi'_{\pm}$.}
\end{equation}
(If $\inf(\sigma(T)) = 0$, one has $\psi_{T,+}(0,x) = \psi_{T,-}(0,x)$ and hence
$\phi_+ = \phi_-$.)
\end{proof}

Given Hypothesis \ref{h4.1}, Lemma \ref{l4.2} guarantees the existence of a real-valued,
$\omega$-periodic $\phi \in L^2_{\locunif}(\bbR;dx)$ such that
$q = \phi^2 - \phi'$ and hence we can identify the operator $T$ in $L^2(\bbR; dx)$ with
$T_1 = A^* A$ in \eqref{3.21} (resp., \eqref{3.46}), where $A$ and $A^*$ defined as in
\eqref{3.22} and \eqref{3.23} (resp., \eqref{3.46}). In addition, we define the periodic Dirac-type
operator $D$ in $L^2(\bbR; dx) \oplus L^2(\bbR; dx)$ by \eqref{3.49}.

Since $\phi \in L^2([0,\omega]; dx)$, for any $\varepsilon>0$ and all $g \in H^1((0,\omega))$,
one has
\begin{align}
\begin{split} 
\|\phi g\|^2_{L^2([0,\omega]; dx)} & \leq \varepsilon \|g' \|^2_{L^2([0,\omega]; dx)}   \\
& \quad + \|\phi\|^2_{L^2([0,\omega]; dx)} \big[\omega^{-1} 
+ \|\phi\|^2_{L^2([0,\omega]; dx)} \varepsilon^{-1}\big] \|g \|^2_{L^2([0,\omega]; dx)}    \lb{4.1}
\end{split} 
\end{align}
(cf.\ \cite[p.\ 19--20, 37]{Sc81}). Utilizing \eqref{4.1}, one can introduce the reduced Dirac-type
operator $D_{\theta}$ in $L^2([0, \omega]; dx)$, $\theta \in [0, 2\pi]$, by
\begin{equation}
D_{\theta} = \begin{pmatrix} 0 & A^*_{\theta} \\ A^{}_{\theta} & 0 \end{pmatrix}
\, \text{ in $L^2([0, \omega]; dx) \oplus L^2([0, \omega]; dx)$,}     \lb{4.2}
\end{equation}
where
\begin{align}
& A^{}_{\theta} = (d/dx) + \phi,   \quad
\dom(A^{}_{\theta}) = \big\{g \in H^1((0,\omega)) \,\big|\,
g(\omega) = e^{i \theta} g(0)\big\},   \lb{4.3} \\
& A^*_{\theta} = -(d/dx) + \phi,   \quad
\dom(A^*_{\theta}) = \big\{g \in H^1((0,\omega)) \,\big|\,
g(\omega) = e^{i \theta} g(0)\big\},    \lb{4.4}
\end{align}
and $A^{}_{\theta}$ (and hence $A^*_{\theta}$) is closed in $L^2([0,\omega]; dx)$, implying
self-adjointness of $D_{\theta}$.

Employing the identity \eqref{3.50}, $D^2 = T_1 \oplus T_2$, and analogously for $D^{2}_{\theta}$,
\begin{align}
& D_{\theta}^2 = \begin{pmatrix} A^*_{\theta} A^{}_{\theta} & 0 \\ 0 & A^{}_{\theta} A^*_{\theta}
\end{pmatrix}
= T_{1,\theta} \oplus T_{2,\theta} \, \text{ in $L^2([0,\omega]; dx) \oplus L^2([0,\omega]; dx)$,}
\lb{4.5} \\
& T_{1,\theta} = A^*_{\theta} A^{}_{\theta},  \quad T_{2,\theta} = A^{}_{\theta} A^*_{\theta}
 \, \text{ in $L^2([0,\omega]; dx)$,}     \lb{4.6}
\end{align}
and applying the standard direct integral formalism combined with Floquet theory to $D$,
$D_{\theta}$ (cf., \cite[App.\ to Ch.\ 10]{Be82}, \cite{Ea73}, \cite[Sect.\ XIII.16]{RS78}),
where
\begin{equation}
L^2(\bbR; dx) \simeq \f{1}{2 \pi} \int^{\oplus}_{[0, 2\pi]} d \theta \, L^2([0,\omega]; dx),  \lb{4.7}
\end{equation}
then yields the following result (with $\simeq$ abbreviating unitary equivalence):

\begin{theorem} \lb{t4.3}
Assume Hypothesis \ref{h4.1}. Then the periodic Dirac operator $D$ $($cf.\ \eqref{3.49}$)$ satisfies 
\begin{equation}
D \simeq \f{1}{2 \pi} \int^{\oplus}_{[0, 2\pi]} d \theta \, D^{}_{\theta},
\end{equation}
with respect to the direct integral decomposition \eqref{4.7}, and
$\sigma_{p}(D) = \sigma_{sc}(D) = \emptyset$. Moreover, $\sigma(D)$ is purely absolutely
continuous of uniform spectral multiplicity equal to two, and $\sigma(D)$ consists of a union
of compact intervals accumulating at $+\infty$ and $-\infty$.

In addition, the spectra of $T_j$ $($cf.\ \eqref{3.47}, \eqref{3.70}$)$ satisfy 
$\sigma_{p}(T_j) = \sigma_{sc}(T_j) = \emptyset$,
in fact, $\sigma(T_j)$ is purely absolutely continuous of uniform spectral multiplicity equal
to two, and $\sigma(T_j)$ consists of a union of compact intervals accumulating at $+\infty$,
$j=1,2$.
\end{theorem}

We note in passing that the spectral properties of $T_j$, $j=1,2$, alternatively, also follow from the
$m$-function relations \eqref{3.57}, \eqref{3.58}. In fact, applying the results in \cite{EGNT12}, one
can extend Theorem \ref{t4.3} to the case where $\phi \in L^1_{\loc}(\bbR; dx)$ is real-valued and
periodic of period $\omega > 0$, but we will not pursue this any further in this paper.

The supersymmetric approach linking (periodic, quasi-periodic, finite-gap, etc.) Schr\"odinger
and Dirac-type operators has been applied repeatedly in the literature, see, for instance,
\cite{DM05}, \cite{Ge91}, \cite{Ge92}, \cite{GSS91}, \cite{GS95}, \cite{Ko00},
and the extensive literature cited therein. In addition, we note that spectral theory (gap and eigenvalue asymptotics, etc.) for
Schr\"odinger operators with periodic distributional potentials has been thoroughly investigated in
\cite{DM09}, \cite{DM09a}, \cite{DM10}, \cite{DM12}, \cite{HM01}, \cite{HM02}, \cite{KM01},
\cite{Ka10a}, \cite{Ko03}, \cite{Ko12}, \cite{MM08}, \cite{MM09}.

We now investigate the eigenvalues associated with the differential equation \eqref{1.7} and quasi-periodic boundary conditions utilizing the operator $\ul{T^{-1/2} r T^{-1/2}}$ in $L^2([0,\omega]; dx)$ when $r$ is a measure. More precisely, let $R:[0,\omega]\to\bbR$ be a left-continuous real-valued function of bounded variation and $\mu_R$ the associated signed measure. We associate with $R$ the following map
\begin{equation} \label{r4.21}
r: H^1((0,\omega))\to H^{-1}((0,\omega))
\end{equation}
via the Lebesgue--Stieltjes integral,
\begin{equation} \label{r4.22}
\dpl r f, g \dpr = \int_0^\omega d\mu_R (x) \, \ol{f(x)} g(x), \quad f, g \in H^1((0,\omega)).
\end{equation}
One notes that the map $r$ in \eqref{r4.21} is bounded.

We also write
$R=R_+-R_-$ where $R_\pm$ are both left-continuous and nondecreasing and thus give rise to positive finite measures on $[0,\omega]$.

Thus,
\begin{equation}
K \in \cB\big(L^2([0,\omega]; dx), H^1((0,\omega))\big) \, \text{ implies } \,
K^*rK \in \cB\big(L^2([0,\omega]; dx)\big).     \lb{4.29}
\end{equation}
Similarly,
\begin{equation}
K \in \cB_{\infty}\big(L^2([0,\omega]; dx), H^1((0,\omega))\big) \, \text{ implies } \,
K^*rK \in \cB_{\infty}\big(L^2([0,\omega]; dx)\big).    \lb{4.30}
\end{equation}

\begin{lemma} \lb{l4.6}
Suppose $K \in \cB_\infty(\big(L^2([0,\omega]; dx),H^1((0,\omega))\big)$ is compact and that
$C_0^\infty((0,\omega)) \subset \ran(K)$. In addition, assume that $R$ is a real-valued function of bounded variation on $[0,\omega]$ and define $r$ as in \eqref{r4.21}, \eqref{r4.22}. Then
$K^* r K$ has infinitely many positive $($resp., negative$)$ eigenvalues unless
$R_+$ $($resp., $R_-$$)$ is a pure jump function with only finitely many jumps $($if any$)$.
\end{lemma}
\begin{proof}
Without loss of generality we may assume that $\omega=1$ and we may also restrict
attention to $R_+$ only. Accordingly, suppose that the measure associated with $R_+$ has a continuous part or that $R_+$ has
infinitely many jumps, but, that by way of contradiction, $K^* r K$ has only finitely many
(say, $N \geq 0$) positive eigenvalues. We will show below that there is a positive number
$\ell$ and $N+1$ sets $\Omega_1, \dots, \Omega_{N+1}$, which have a distance of at least
$\ell$ from each other and from the endpoints of $[0,1]$, for which $\int_{\Omega_j} d\mu_R>0$.

For any $\varepsilon$, with $0<\varepsilon<\ell/2$, let $J_{\varepsilon}$ be the Friedrichs
mollifier as introduced, for instance, in \cite[Sect.\ 2.28]{AF03}. Applying \cite[Theorem\ 2.29]{AF03},
the functions
\begin{equation}
g_{j,\varepsilon} = J_\varepsilon*\chi_{\Omega_j}, \quad j = 1, \dots, N+1,   \lb{4.12a}
\end{equation}
satisfy the following properties: \\[1mm]
$(i)$ $g_{j,\varepsilon} \in C_0^\infty((0,1)) \subset \ran(K)$,    \\[1mm]
$(ii)$ $g_{j,\varepsilon}$ are zero at points which are further than $\varepsilon$ away
from $\Omega_j$, \\[1mm]
$(iii)$ $\lim_{\varepsilon \downarrow 0}\|g_{j,\varepsilon}-\chi_{\Omega_j}\|_{L^2([0,1]; dx)} = 0$,   \\[1mm]
$(iv)$ $|g_{j,\varepsilon}(x)| \leq 1$.  \\[1mm]
Property $(i)$ implies that there are functions $f_{j,\varepsilon} \in L^2([0,1]; dx)$ such that
$g_{j,\varepsilon} =Kf_{j,\varepsilon}$ since $C_0^\infty((0,1)) \subset \ran(K)$. By property $(iii)$,
$g_{j,\varepsilon} \to \chi_{\Omega_j}$ pointwise a.e.\ on $(0,1)$ as $\varepsilon \downarrow 0$,
and hence the dominated convergence theorem implies that
\begin{equation}
\dpl r Kf_{j,\varepsilon}, Kf_{j,\varepsilon} \dpr 
= \int_{[0,1]} d\mu_R(x) \, |g_{j,\varepsilon}(x)|^2
\underset{\varepsilon \downarrow 0}{\longrightarrow} \int_{\Omega_j} d\mu_R(x) > 0.
\end{equation}
Hence we may fix $\varepsilon > 0$ in such a way that
\begin{equation}
\int_{[0,1]} d\mu_R(x) \, |(Kf_{j,\varepsilon})(x)|^2 > 0, \quad  j=1, \dots , N+1.    \lb{4.13A}
\end{equation}
Next, by property $(ii)$ mentioned above, the supports of the
$g_{j,\varepsilon}$ are pairwise disjoint, implying
\begin{equation}
\bigg|\sum_{j=1}^{N+1} c_j g_{j,\varepsilon}\bigg|^2=\sum_{j+1}^{N+1}|c_j|^2 |g_{j,\varepsilon}|^2
\lb{4.14A}
\end{equation}
for any choice of $c_j \in \bbC$, $j = 1, \dots, N+1$.

Assume now that $f=\sum_{j=1}^{N+1} c_j f_{j,\varepsilon}$, where at least one of the
coefficients $c_j \neq 0$. Then equations \eqref{4.13A} and \eqref{4.14A} imply
\begin{align}
\begin{split}
(f, K^* r K f)_{L^2([0,1]; dx)} &= \int_{[0,1]} d\mu_R(x) \, |(Kf)(x)|^2    \\
&=\sum_{j=1}^{N+1}|c_j|^2 \int_{[0,1]} d\mu_R(x) \, |(Kf_{j,\varepsilon})(x)|^2 > 0. \label{4.17A}
\end{split}
\end{align}

We will now prove that for some choices of the coefficients $c_j$, the expression
$(f,K^*rK  f)_{L^2([0,1]; dx)}$ cannot be positive so that one arrives at a contradiction to
\eqref{4.17A}, proving that there must be infinitely many positive eigenvalues. To do so,
we denote the nonzero eigenvalues and eigenfunctions of the compact, self-adjoint operator
$K^* r K$ by $\lambda_k$ and $\varphi_k$,
respectively. More specifically, assume that the positive eigenvalues have labels $k=1,\dots, N$,
while the labels of the non-positive eigenvalues are chosen from the non-positive integers.
The spectral theorem, applied to $K^* r K$, yields
\begin{align}
\begin{split}
0< (f,K^* r K f)_{L^2([0,1]; dx)} &= \sum_{k=-\infty}^N \lambda_k |(\varphi_k,f)_{L^2([0,1]; dx)}|^2  \\
&\leq \sum_{k=1}^N \lambda_k |(\varphi_k,f)_{L^2([0,1]; dx)}|^2 \lb{4.18A}
\end{split}
\end{align}
for any $f\in L^2([0,1]; dx)$. If $N=0$, this is the desired contradiction. If $N\geq1$, the
inequality \eqref{4.18A} shows that no non-zero element of $L^2([0,1]; dx)$ can be orthogonal
to all the eigenfunctions associated with positive eigenvalues. However,
the underdetermined system
\begin{equation}
\sum_{j=1}^{N+1} c_j (\varphi_k,f_{j,\varepsilon})_{L^2([0,1]; dx)}
= (\varphi_k,f)_{L^2([0,1]; dx)} = 0, \quad k=1, \dots, N,
\end{equation}
has nontrivial solutions $(c_1,\dots,c_N)$ proving that $f=\sum_{j=1}^{N+1} c_j f_{j,\varepsilon}$
is orthogonal to all the eigenfunctions associated with positive eigenvalues so that we again arrive at a contradiction.

It remains to establish the existence of the sets $\Omega_j$ with the required properties.
Recall that, by Lebesgue's decomposition theorem, $R=R_1+R_2+R_3$, where $R_1$ is absolutely continuous, $R_2$ is
continuous but $R_2'=0$ a.e.\ on $[0,1]$, and $R_3$ is a jump function and that these
generate an absolutely continuous measure $\mu_1$, a singular continuous measure 
$\mu_2$, and a discrete measure $\mu_3$ (i.e., one supported on a countable subset 
of $\bbR$), respectively.
By Jordan's decomposition theorem, each of these measures may be split into its positive and negative part $\mu_{j,\pm}$, $j=1,2,3$.
We will denote the respective supports of these measures by $A_{j,\pm}$, $j=1,2,3$.
Note that $A_{j,+}\cap A_{j,-}$ is empty for each $j$ by Hahn's decomposition theorem.
We also define $R_{j,\pm}(x)=\mu_{j,\pm}([0,x])$.

First, we assume that the support $A_{1,+}$ of $\mu_{1,+}$ has positive Lebesgue measure.
Since the supports of $\mu_2$ and $\mu_3$ have zero Lebesgue measure, they are subsets
of a union of open intervals whose total length is arbitrarily small.
Thus, we may find a set $\Omega\subset A_{1,+}$ of positive Lebesgue measure
which avoids a neighborhood of the supports of $\mu_2$ and $\mu_3$ so that $\int_\Omega d\mu_R>0$.
Now define $M = \lceil (2N+3)/m(\Omega)\rceil$, with $m(\cdot)$ abbreviating Lebesgue measure and $\lceil x \rceil$ the
smallest integer not smaller than $x$.
Dividing the interval $[0,1]$ uniformly into $M$ subintervals, each will have length not
exceeding $\ell = m(\Omega)/(2N+3)$.
Consequently, at least $2N+3$ of these intervals will intersect $\Omega$ in a set of positive
Lebesgue measure and hence of positive $\mu_R$-measure. $N+1$ of the latter ones will have
a distance of at least $\ell$ from each other and from the endpoints of $[0,1]$.
These intersections will be the sought after sets $\Omega_1, \dots, \Omega_{N+1}$.

Next assume $\mu_{1,+}=0$ but $\mu_{2,+}([0,1])=a_2>0$.
Since $A_{3,-}$ is countable we have $\mu_{2,+}(A_{3,-})=0$. Also, of course, $\mu_{2,+}(A_{2,-})=0$.
By the regularity of $\mu_{2,+}$ there is, for every positive $\varepsilon$, an open set $W$ covering $A_{2,-}\cup A_{3,-}$ such that $\mu_{2,+}(W)<\varepsilon$.
Set $\Omega=(0,1)\backslash \ol{W}$ and $\varepsilon=a_2/2$. Since $\ol{W}-W$ is countable we have $\mu_{2,+}(\Omega)=\mu_{2,+}((0,1)\backslash W)>a_2/2$.
Since $R_{2,+}$ is uniformly continuous there is a $\delta>0$ so that $R_{2,+}(y)-R_{2,+}(x)<a_2/(2(2N+3))$ as long as $0<y-x<\delta$.
Thus, splitting $\Omega$ in intervals of length at most $\delta$, we have that at least $2N+3$ of these intervals have positive $\mu_{2,+}$-measure and $N+1$ of these have a positive distance from each other and from the endpoints of $[0,1]$.
We denote these intervals by $\Omega_1'$, ..., $\Omega_{N+1}'$.
We now have $\mu_{2,+}(\Omega_k')>0$ but $\mu_{2,-}(\Omega_k')=\mu_{3,-}(\Omega_k')=0$.
However, it may still be the case that $\mu_{1,-}(\Omega_k')>\mu_{2,+}(\Omega_k')$.
Regularity of $\mu_{1,-}$ allows us to find a set $\Omega_k$ such that $A_{2,+}\cap \Omega_k' \subset \Omega_k \subset \Omega_k'$ and $\mu_{1,-}(\Omega_k)$ are arbitrarily small.
This way we may guarantee that $\mu(\Omega_k)>0$ for $k=1,... N+1$.

Finally, assume that $R_+$ is a pure jump function, but with infinitely many jumps.
Then we may choose pairwise disjoint intervals $\Omega_k$ about $N+1$ of the jump discontinuities of
$R_+$ and we may choose them so small that their $\mu_{j,-}(\Omega_k)$ is smaller than the jump so that again $\mu(\Omega_k)>0$ for $k=1,... N+1$.
\end{proof}

We emphasize that Lemma \ref{l4.6} applies, in particular, to the special case, where 
$d \mu_R (x) = r(x) dx$ is purely absolutely continuous on $\bbR$:

\begin{corollary} \lb{c4.4}
Suppose $K \in \cB_{\infty}\big(L^2([0,\omega]; dx)\big)$ is self-adjoint with 
$\ran(K) \supseteq H^1((0,\omega))$. Assume in addition that $r \in L^1([0,\omega]; dx)$ is real-valued such that $|r|^{1/2}K \in \cB_{\infty}\big(L^2([0,\omega]; dx)\big)$. Then 
$\ul{K r K} := [|r|^{1/2}K]^* \sgn(r) |r|^{1/2}K$ 
has infinitely many positive $($resp., negative$)$ eigenvalues unless $r_+=0$ 
$($resp., $r_-=0$$)$ a.e.\ on $(0,\omega)$. 
\end{corollary}

Identifying $T_{\theta}$ in $L^2([0,\omega]; dx)$ with $T_{1,\theta} = A^*_{\theta} A^{}_{\theta}$
(in analogy to the identification of $T$ in $L^2(\bbR; dx)$ with $T_1 = A^* A$),
an application of Lemma \ref{l4.6}, employing \eqref{C.41}--\eqref{C.41e}, then yields the 
following result:

\begin{theorem} \lb{t4.5}
Assume that $\mu_R$ is a signed measure and let $r$ be defined as in \eqref{r4.21}, \eqref{r4.22}. 
Assume $r$ is periodic of period $\omega > 0$. \\
$(i)$ Suppose that $T_{\theta} \geq c_{\theta} I_{L^2([0,\omega]; dx)}$ for some
$c_{\theta} > 0$. Then 
$\big({\wti T_{\theta}}\big)^{-1/2} r \big({\wti T_{\theta}}\big)^{-1/2}$
has infinitely many positive $($resp., negative$)$ eigenvalues unless $R_+$
$($resp., $R_-$$)$ is a pure jump function with only finitely many jumps $($if any\,$)$. \\
$(ii)$ Suppose that $T \geq c I_{L^2(\bbR; dx)}$ for some $c > 0$. Then
$\sigma\Big(\big({\wti T}\big)^{-1/2} r \big({\wti T}\big)^{-1/2}\Big)$ consists of a union of compact intervals accumulating
at $0$ unless $r=0$ a.e.\ on $(0,\omega)$. In addition,
\begin{equation} \label{r4.29}
- \psi'' + q \psi = z r \psi
\end{equation}
has a conditional stability set $($consisting of energies $z$ with at least
one bounded solution on $\bbR$$)$ composed of a sequence of intervals on $(0,\infty)$
tending to $+\infty$ and/or $-\infty$, unless $R_+$ and/or $R_-$ is a pure jump function with 
only finitely many jumps $($if any\,$)$. Finally,
\begin{equation}
\sigma_{\rm p} \Big(\big({\wti T}\big)^{-1/2} r \big({\wti T}\big)^{-1/2}\Big) = \emptyset.     \lb{4.32} 
\end{equation}
\end{theorem}
\begin{proof}
Lemma \ref{l4.6}, identifying $K$ and $\big({\wti T_{\theta}}\big)^{-1/2}$ (cf.\ \eqref{C.41d} and our 
notational convention \eqref{C.41e}) proves item $(i)$. 

As usual (see Eastham \cite[Sect.\ 2.1]{Ea73} or Brown, Eastham, and Schmidt 
\cite[Sect.\ 1.4]{BES12}), the conditional stability set $S$ of equation \eqref{r4.29} is given by 
\begin{equation}
S=\{\lambda \in \bbR \, | \, |\tr_{\bbC^2} (M(\lambda))| \leq 2\}
\end{equation}
since, if $\lambda \in S$ and only then, the monodromy operator $M(\lambda)$ has at least one eigenvector associated with an eigenvalue of modulus $1$. Since $\tr_{\bbC^2} (M(\cdot))$ is 
an analytic, hence, continuous function, the set 
$S^{0}=\{\lambda \in \bbR \, | \, |\tr_{\bbC^2} (M(\lambda))|<2\}$ is an open set 
and thus a union of open intervals. Moreover, 
$\{\lambda \in \bbR \, | \, \tr_{\bbC^2} (M(\lambda))=2\}$ (i.e., the set of periodic eigenvalues) and $\{\lambda \in \bbR \, | \, \tr_{\bbC^2} (M(\lambda))=-2\}$ (i.e., the set of anti-periodic eigenvalues) are discrete sets without finite accumulation points. It follows that $S$ is obtained as the union of the closures of each of the open intervals constituting $S^{0}$, equivalently, $S$ is a union of closed intervals. One notes that the closure of several disjoint components of $S^{0}$ may form one closed interval in $S$.

Applying Lemma \ref{l4.6} to the case $K = \big({\wti T_{\theta}}\big)^{-1/2}$ one obtains a countable 
number of eigenvalues $\zeta_n(\theta)$, $n\in\bbZ \backslash \{0\}$ which we 
may label so that $n \, \zeta_n(\theta)>0$. These eigenvalues accumulate at zero (from either 
side). It is clear that equation \eqref{r4.29} posed on the interval $[0,\omega]$ 
has a nontrivial solution satisfying the boundary conditions 
$\psi(\omega)=e^{i\theta}\psi(0)$ and $\psi^{[1]}(\omega)=e^{i\theta}\psi^{[1]}(0)$ precisely when $z=1/\zeta_n(\theta)$ for some $n\in\bbZ \backslash \{0\}$. In particular, the endpoints of the 
conditional stability intervals, which correspond to the values 
$\theta=0$ and $\theta=\pi$, tend to both, $+\infty$ and $-\infty$. 

Finally, eigenfunctions $u \in L^2(\bbR; dx)$ of $\big({\wti T}\big)^{-1/2} r \big({\wti T}\big)^{-1/2}$ 
are related to solutions $y \in H^1(\bbR)$ of 
\eqref{1.7} (with $p=1$) via $y = \big({\wti T}\big)^{-1/2} u$. Since the basics of Floquet theory 
apply to \eqref{1.7} (cf.\ our comments at the beginning of this section and earlier in the current 
proof), the existence of Floquet multipliers $\rho(z)$ and $1/\rho(z)$ prevents \eqref{1.7} from 
having an $L^2(\bbR; dx)$ (let alone, $H^1(\bbR)$) solution. Hence, the existence of an 
eigenfunction $u \in L^2(\bbR; dx)$ of $\big({\wti T}\big)^{-1/2} r \big({\wti T}\big)^{-1/2}$ would imply 
the contradiction $y \in H^1(\bbR)$, implying \eqref{4.32}.
\end{proof}

Theorem \ref{t4.5} considerably extends prior results by Constantin \cite{Co97} (see also
\cite{Co97a}, \cite{Co98}) on eigenvalue asymptotics for left-definite periodic
Sturm--Liouville problems since no smoothness is assumed on $q$ and $r$, in 
addition, $q$ is permitted to be a distribution and $r$ is extended from merely being a  
function to a measure. Moreover, it also 
extends results of Daho and Langer \cite{DL86}, Marletta and Zettl \cite{MZ05}, and
Philipp \cite{Ph13}: While these authors consider the nonsmooth setting, our result appears 
to be the first that permits periodic distributions, respectively, measures as  coefficients.

\begin{remark} \lb{r4.7}
In the special case where the measure $d \mu_R (x) = r(x) dx$ is purely absolutely continuous 
on $\bbR$, the fact that
\begin{equation}
\ul{T^{-1/2} r T^{-1/2}} \simeq \f{1}{2 \pi} \int^{\oplus}_{[0,2 \pi]}
\ul{T^{-1/2}_{\theta} r T^{-1/2}_{\theta}}
\end{equation}
with respect to the decomposition \eqref{4.7}, together with continuity of the eigenvalues
of $\ul{T^{-1/2}_{\theta} r T^{-1/2}_{\theta}}$ with respect to $\theta$, proves that
$\sigma(\ul{T^{-1/2} r T^{-1/2}})$ consists of a union of compact intervals accumulating
at $0$ unless $r=0$ a.e.\ on $(0,\omega)$. 

Moreover, employing the methods in \cite[Sect.\ 2]{GN12}, Theorem \ref{t4.5}\,$(i)$ 
immediately extends to any choice of self-adjoint separated boundary conditions replacing 
the $\theta$ boundary conditions
\begin{equation}
g(\omega) = e^{i \theta} g(0), \quad g'(\omega) = e^{i \theta} g'(0),  \quad \theta \in [0, 2 \pi],
\end{equation}
in $A^*_{\theta} A_{\theta}$ by separated ones of the type 
\begin{equation}
\sin(\alpha) g'(0) + \cos(\alpha) g(0) =0, \quad \sin(\beta) g'(\omega) + \cos(\beta) g(\omega) =0,
\quad \alpha, \beta \in [0,\pi].
\end{equation}
\end{remark} 

\smallskip

We emphasize that the following Appendices \ref{sA}, \ref{sB}, and \ref{sC} do not 
contain new results. We offer them for the convenience of the reader with the 
goal of providing a fairly self-contained account, enhancing the readability of this manuscript.

\appendix
\section{Relative Boundedness and Compactness \\ of Operators and
Forms} \lb{sA}
\renewcommand{\theequation}{A.\arabic{equation}}
\renewcommand{\thetheorem}{A.\arabic{theorem}}
\setcounter{theorem}{0} \setcounter{equation}{0}

In this appendix we briefly recall the notion of relatively bounded (resp., compact) and relatively form bounded (resp., form compact) perturbations of a self-adjoint operator $A$ in some complex separable Hilbert space $\cH$:

\begin{definition} \lb{dA.1}
$(i)$ Suppose that $A$ is a self-adjoint operator in $\cH$. A closed operator $B$ in $\cH$ is
called {\it relatively bounded} $($resp., {\it relatively compact}\,$)$ {\it with respect to $A$}
$($in short, $B$ is called {\it $A$-bounded} $($resp., {\it $A$-compact}\,$))$, if
\begin{equation}
 \dom(B) \supseteq \dom(A) \, \text{ and } \,  B(A - z I_{\cH})^{-1} \in \cB(\cH) \;
 (\text{resp.,} \in \cB_\infty(\cH)), \quad  z \in \rho(A).   \lb{A.1}
\end{equation}
$(ii)$ Assume that $A$ is self-adjoint and bounded from below $($i.e., $A \ge c I_{\cH}$ for some
$c\in\bbR$$)$. Then a densely defined and closed operator $B$ in $\cH$ is called {\it relatively form bounded} $($resp., {\it relatively form compact}\,$)$ {\it with respect to $A$} $($in short, $B$ is
called {\it $A$-form bounded} $($resp., {\it $A$-form compact}\,$))$, if
\begin{align}
\begin{split}
 & \dom\big(|B|^{1/2}\big) \supseteq \dom\big(|A|^{1/2}\big) \, \text{ and } \, \\
 & |B|^{1/2} ((A + (1 - c) I_{\cH}))^{-1/2} \in \cB(\cH) \;
 (\text{resp.,} \in \cB_\infty(\cH)).
 \lb{A.2}
 \end{split}
\end{align}
\end{definition}

\medskip

\begin{remark} \lb{rA.2}
$(i)$ Using the polar decomposition of $B$ (i.e., $B=U_B|B|$, with $U_B$ a partial isometry), one observes that $B$ is $A$-bounded (resp., $A$-compact) if and only if $|B|$ is $A$-bounded (resp., $A$-compact). Similarly, by \eqref{A.2}, $B$ is
$A$-form bounded (resp., $A$-form compact), if and only if $|B|$ is. \\
$(ii)$ Since$B$ is assumed to be closed (in fact, closability of $B$ suffices) in
Definition \ref{dA.1}\,$(i)$, the first condition $\dom(B) \supseteq \dom(A)$ in
\eqref{A.1} already implies $B(A - z I_{\cH})^{-1} \in \cB(\cH)$, $z \in \rho(A)$,
and hence the $A$-boundedness of $B$ (cf.\ again \cite[Remark IV.1.5]{Ka80},
\cite[Theorem\ 5.9]{We80}). By the same token, since $A^{1/2}$ and $|B|^{1/2}$ are closed, the requirement
$ \dom\big(|B|^{1/2}\big) \supseteq \dom\big(A^{1/2}\big)$ in
Definition \ref{dA.1}\,$(ii)$,  already implies that
$|B|^{1/2} ((A + (1 - c) I_{\cH}))^{-1/2} \in \cB(\cH)$
(cf.\ \cite[Remark IV.1.5]{Ka80}, \cite[Theorem\ 5.9]{We80}), and hence the first condition in \eqref{A.2} suffices in the relatively form bounded context. \\
$(iii)$ In the special case where $B$ is self-adjoint, condition
\eqref{A.2} implies the existence of $\alpha\geq 0$ and $\beta\geq 0$, such that
\begin{align}
\begin{split}
\big|\big(|B|^{1/2}f, \sgn(B) |B|^{1/2}f\big)_{\cH}\big|
\leq \big\||B|^{1/2}f\big\|_{\cH}^2
\leq \alpha \big\||A|^{1/2}f\big\|_{\cH}^2 + \beta \|f\|_{\cH}^2,& \\
f\in\dom\big(|A|^{1/2}\big).&    \lb{A.3}
\end{split}
\end{align}
$(iv)$ In connection with relative boundedness, \eqref{A.1} can be replaced by the condition
\begin{align}
\begin{split}
& \dom(B) \supseteq \dom(A), \, \text{ and there exist numbers
$a\ge 0$, $b\ge 0$ such that} \\
& \|Bf\|_{\cH} \le a \|Af\|_{\cH} + b \|f\|_{\cH} \, \text{ for all $f\in\dom(A)$,}
\lb{A.4}
\end{split}
\end{align}
or equivalently, by
\begin{align}
\begin{split}
& \dom(B) \supseteq \dom(A), \, \text{ and there exist numbers $\wti a\ge 0$, $\wti b\ge 0$ such that} \\
& \|Bf\|^2_{\cH} \le {\wti a}^2 \|Af\|^2_{\cH} + {\wti b}^2 \|f\|^2_{\cH} \, \text{ for all $f\in\dom(A)$.}     \lb{A.5}
\end{split}
\end{align}
$(v)$ If $A$ is self-adjoint and bounded from below, the number $\alpha$ defined by
\begin{equation}
\alpha = \lim_{\mu\uparrow \infty} \big\|B(A+\mu I_{\cH})^{-1}\big\|_{\cB(\cH)}
= \lim_{\mu\uparrow \infty} \big\||B|(A+\mu I_{\cH})^{-1}\big\|_{\cB(\cH)}
\lb{A.6}
\end{equation}
equals the greatest lower bound (i.e., the infimum) of the possible values for $a$ in \eqref{A.4} (resp., for
$\wti a$ in \eqref{A.5}). This number $\alpha$ is called the $A$-bound of $B$. Similarly, we call
\begin{equation}
\beta =
\lim_{\mu\uparrow \infty} \big\||B|^{1/2}\big(|A|^{1/2}+\mu I_{\cH}\big)^{-1}\big\|_{\cB(\cH)}     \lb{A.7}
\end{equation}
the $A$-form bound of $B$ (resp., $|B|$). If $\alpha =0$ in \eqref{A.6} (resp., $\beta =0$ in \eqref{A.7}) then $B$ is called {\it infinitesimally bounded} (resp., {\it infinitesimally form bounded}\,) with respect to $A$.
\end{remark}

We then have the following result:

\begin{theorem} \lb{tA.3}
Assume that $A\ge 0$ is self-adjoint in $\cH$. \\
$(i)$ Let $B$ be a closed, densely defined operator in $\cH$ and suppose that
$\dom(B)\supseteq\dom(A)$. Then $B$ is $A$-bounded and hence \eqref{A.4} holds for some constants $a\geq 0$, $b\geq 0$. In addition, $B$ is also $A$-form bounded,
\begin{equation}
 |B|^{1/2} (A +I_{\cH})^{-1/2} \in \cB(\cH).
\lb{A.8}
\end{equation}
More specifically,
\begin{equation}
\big\||B|^{1/2}(A+I_{\cH})^{-1/2}\big\|_{\cB(\cH)} \le (a+b)^{1/2},    \lb{A.9}
\end{equation}
and hence, if $B$ is $A$-bounded with $A$-bound $\alpha$ strictly less than one, $0\leq \alpha<1$ $($cf.\ \eqref{A.6}$)$, then $B$ is also $A$-form bounded with $A$-form bound $\beta$ strictly less than one, $0\leq \beta <1$ $($cf.\ \eqref{A.7}$)$. In particular, if $B$ is infinitesimally bounded with respect to $A$, then $B$ is infinitesimally form bounded with respect to $A$. \\
$(ii)$ Suppose that $B$ is self-adjoint in $\cH$, that
$\dom(B) \supseteq\dom(A)$, and hence \eqref{A.4} holds for some constants
$a\geq 0$, $b\geq 0$. Then
\begin{align}
& \ol{(A +  I_{\cH})^{-1/2} B (A +  I_{\cH})^{-1/2}} \in \cB(\cH),   \lb{A.10}  \\
& \big\|\ol{(A+I_{\cH})^{-1/2}B(A+I_{\cH})^{-1/2}}\big\|_{\cB(\cH)}
 \le (a+b).   \lb{A.11}
\end{align}
\end{theorem}

We also recall the following result:

\begin{theorem} \lb{tA.4}
Assume that $A\ge 0$ is self-adjoint in $\cH$. \\
$(i)$ Let $B$ be a densely defined closed operator in $\cH$ and suppose that
$\dom(B)\supseteq\dom(A)$. In addition, assume that $B$ is $A$-compact. Then $B$ is also $A$-form compact,
\begin{equation}
|B|^{1/2} (A+I_{\cH})^{-1/2} \in\cB_\infty (\cH).     \lb{A.12}
\end{equation}
$(ii)$ Suppose that $B$ is self-adjoint in $\cH$ and that
$\dom(B) \supseteq \dom(A)$. In addition, assume that $B$ is
$A$-compact. Then
\begin{equation}
 \ol{(A+I_{\cH})^{-1/2} B (A+I_{\cH})^{-1/2}}
  \in \cB_\infty (\cH).     \lb{A.13}
\end{equation}
\end{theorem}

For proofs of Theorems \ref{tA.3} and \ref{tA.4} under more general conditions on $A$ and $B$,
we refer to \cite{GMMN09} and the detailed list of references therein.

\section{Supersymmetric Dirac-Type Operators in a Nutshell} \lb{sB}
\renewcommand{\theequation}{B.\arabic{equation}}
\renewcommand{\thetheorem}{B.\arabic{theorem}}
\setcounter{theorem}{0} \setcounter{equation}{0}

In this appendix we briefly summarize some results on supersymmetric
Dirac-type operators and commutation methods due to \cite{De78},
\cite{GSS91}, \cite{Th88}, and \cite[Ch.\ 5]{Th92} (see also \cite{HKM00}).

The standing assumption in this appendix will be the following.

\begin{hypothesis} \lb{hB.1}
Let $\cH_j$, $j=1,2$, be separable complex Hilbert spaces and
\begin{equation}
A: \cH_1 \supseteq \dom(A) \to \cH_2    \lb{B.1}
\end{equation}
be a densely defined, closed, linear operator.
\end{hypothesis}

We define the self-adjoint Dirac-type operator in $\cH_1 \oplus \cH_2$ by
\begin{equation}
Q = \begin{pmatrix} 0 & A^* \\ A & 0 \end{pmatrix}, \quad \dom(Q) = \dom(A) \oplus \dom(A^*).
\lb{B.2}
\end{equation}
Operators of the type $Q$ play a role in supersymmetric quantum mechanics (see, e.g., the extensive list of references in \cite{BGGSS87}). Then,
\begin{equation}
Q^2 = \begin{pmatrix} A^* A & 0 \\ 0 & A A^* \end{pmatrix}     \lb{B.3}
\end{equation}
and for notational purposes we also introduce
\begin{equation}
H_1 = A^* A \, \text{ in } \, \cH_1, \quad H_2 = A A^* \, \text{ in } \, \cH_2.    \lb{B.4}
\end{equation}
In the following, we also need the polar decomposition of $A$ and $A^*$, that is, the representations
\begin{align}
A& = V_A |A| = |A^*| V_A = V_A A^* V_A  \, \text{ on } \, \dom(A) = \dom(|A|),
\label{B.4a} \\
A^*& = V_{A^*} |A^*| = |A|V_{A^*} = V_{A^*} A V_{A^*}   \, \text{ on } \,
\dom(A^*) = \dom(|A^*|),   \label{B.4b} \\
|A|& = V_{A^*} A = A^* V_A = V_{A^*} |A^*| V_A \, \text{ on } \, \dom(|A|),    \label{B.4c} \\
 |A^*|& = V_A A^* = A V_{A^*} = V_A |A| V_{A^*}  \, \text{ on } \, \dom(|A^*|),
 \label{B.4d}
\end{align}
where
\begin{align}
& |A| = (A^* A)^{1/2}, \quad |A^*| = (A A^*)^{1/2},  \quad
V_{A^*} = (V_A)^*,     \lb{B.6a} \\
& V_{A^*} V_A = P_{\ol{{\ran}(|A|)}}=P_{\ol{{\ran}(A^*)}} \, ,  \quad
V_A V_{A^*} = P_{\ol{{\ran}(|A^*|)}}=P_{\ol{{\ran}(A)}} \, .      \label{B.4e}
\end{align}
In particular, $V_A$ is a partial isometry with initial set $\ol{{\ran}(|A|)}$
and final set $\ol{{\ran}(A)}$ and hence $V_{A^*}$ is a partial isometry with initial
set $\ol{\ran(|A^*|)}$ and final set $\ol{\ran(A^*)}$. In addition,
\begin{equation}
V_A = \begin{cases} \ol{A (A^* A)^{-1/2}} = \ol{(A A^*)^{-1/2}A} & \text{on }  (\ker (A))^{\bot},  \\
0 & \text{on }  \ker (A).  \end{cases}     \lb{B.7}
\end{equation}

Next, we collect some properties relating $H_1$ and $H_2$.

\begin{theorem} [\cite{De78}] \lb{tB.2}
Assume Hypothesis \ref{hB.1} and let $\phi$ be a bounded Borel measurable
function on $\bbR$. \\
$(i)$ One has
\begin{align}
& \ker(A) = \ker(H_1) = (\ran(A^*))^{\bot}, \quad \ker(A^*) = \ker(H_2) = (\ran(A))^{\bot},  \lb{B.8} \\
& V_A H_1^{n/2} = H_2^{n/2} V_A, \; n\in\bbN, \quad
V_A \phi(H_1) = \phi(H_2) V_A.  \lb{B.9}
\end{align}
$(ii)$ $H_1$ and $H_2$ are essentially isospectral, that is,
\begin{equation}
\sigma(H_1)\backslash\{0\} = \sigma(H_2)\backslash\{0\},    \lb{B.10}
\end{equation}
in fact,
\begin{equation}
A^* A [I_{\cH_1} - P_{\ker(A)}] \, \text{ is unitarily equivalent to } \, A A^* [I_{\cH_2} - P_{\ker(A^*)}].
\lb{B.10a}
\end{equation}
In addition,
\begin{align}
& f\in \dom(H_1) \, \text{ and } \, H_1 f = \lambda^2 f, \; \lambda \neq 0,   \no \\
& \quad \text{implies }  \,  A f \in \dom(H_2) \, \text{ and } \, H_2(A f) = \lambda^2 (A f),    \lb{B.11} \\
& g\in \dom(H_2)\, \text{ and } \, H_2 \, g = \mu^2 g, \; \mu \neq 0,     \no \\
& \quad \text{implies }  \, A^* g \in \dom(H_1)\, \text{ and } \, H_1(A^* g) = \mu^2 (A^* g),    \lb{B.12}
\end{align}
with multiplicities of eigenvalues preserved. \\
$(iii)$ One has for $z \in \rho(H_1) \cap \rho(H_2)$,
\begin{align}
& I_{\cH_2} + z (H_2 - z I_{\cH_2})^{-1} \supseteq A (H_1 - z I_{\cH_1})^{-1} A^*,    \lb{B.13} \\
& I_{\cH_1} + z (H_1 - z I_{\cH_1})^{-1} \supseteq A^* (H_2 - z I_{\cH_2})^{-1} A,    \lb{B.14}
\end{align}
and
\begin{align}
& A^* \phi(H_2) \supseteq \phi(H_1) A^*, \quad
A \phi(H_1) \supseteq \phi(H_2) A,   \lb {B.15} \\
& V_{A^*} \phi(H_2) \supseteq \phi(H_1) V_{A^*}, \quad
V_A \phi(H_1) \supseteq \phi(H_2) V_A.   \lb {B.15a}
\end{align}
\end{theorem}

As noted by E.\ Nelson (unpublished), Theorem \ref{tB.2} follows from the spectral theorem and the
elementary identities,
\begin{align}
& Q = V_Q |Q| = |Q| V_Q,    \lb{B.16} \\
& \ker(Q) = \ker(|Q|) = \ker (Q^2) = (\ran(Q))^{\bot}
= \ker (A) \oplus \ker (A^*),    \lb{B.17}  \\
& I_{\cH_1 \oplus \cH_2} + z (Q^2 - z I_{\cH_1 \oplus \cH_2})^{-1}
= Q^2 (Q^2 -z I_{\cH_1 \oplus \cH_2})^{-1} \supseteq Q (Q^2 -z I_{\cH_1 \oplus \cH_2})^{-1} Q,  \no \\
& \hspace*{9.3cm}   z \in \rho(Q^2),    \lb{B.18} \\
& Q \phi(Q^2) \supseteq \phi(Q^2) Q,   \lb{B.19}
\end{align}
where
\begin{equation}
V_Q = \begin{pmatrix} 0 & (V_A)^* \\ V_A & 0 \end{pmatrix}
= \begin{pmatrix} 0 & V_{A^*} \\ V_A & 0 \end{pmatrix}.   \lb{B.20}
\end{equation}

In particular,
\begin{equation}
\ker(Q) = \ker(A) \oplus \ker(A^*), \quad
P_{\ker(Q)} = \begin{pmatrix} P_{\ker(A)} & 0 \\ 0 & P_{\ker(A^*)} \end{pmatrix},    \lb{B.21}
\end{equation}
and we also recall that
\begin{equation}
\mathfrak{S}_3 Q \mathfrak{S}_3 = - Q, \quad \mathfrak{S}_3
= \begin{pmatrix} I_{\cH_1} & 0 \\ 0 & - I_{\cH_2}
\end{pmatrix},     \lb{B.22}
\end{equation}
that is, $Q$ and $-Q$ are unitarily equivalent. (For more details on Nelson's
trick see also \cite[Sect.\ 8.4]{Te09}, \cite[Subsect.\ 5.2.3]{Th92}.)
We also note that
\begin{equation}
\psi(|Q|) = \begin{pmatrix} \psi(|A|) & 0 \\ 0 & \psi(|A^*|) \end{pmatrix}    \lb{B.22a}
\end{equation}
for Borel measurable functions $\psi$ on $\bbR$, and
\begin{equation}
\ol{[Q |Q|^{-1}]} = \begin{pmatrix} 0 & (V_A)^*\\ V_A & 0 \end{pmatrix} = V_Q
\, \text{ if } \, \ker(Q) = \{0\}.   \lb{B.22b}
\end{equation}

Finally, we recall the following relationships between $Q$ and $H_j$, $j=1,2$.

\begin{theorem} [\cite{BGGSS87}, \cite{Th88}] \lb{tB.3}
Assume Hypothesis \ref{hB.1}. \\
$(i)$ Introducing the unitary operator $U$ on $(\ker(Q))^{\bot}$ by
\begin{equation}
U = 2^{-1/2} \begin{pmatrix} I_{\cH_1} & (V_A)^* \\ -V_A & I_{\cH_2} \end{pmatrix}
\, \text{ on } \,  (\ker(Q))^{\bot},     \lb{B.23}
\end{equation}
one infers that
\begin{equation}
U Q U^{-1} = \begin{pmatrix}  |A| & 0 \\ 0 & - |A^*| \end{pmatrix}
\, \text{ on } \,  (\ker(Q))^{\bot}.     \lb{B.24}
\end{equation}
$(ii)$ One has
\begin{align}
\begin{split}
(Q - \zeta I_{\cH_1 \oplus \cH_2})^{-1} = \begin{pmatrix} \zeta (H_1 - \zeta^2 I_{\cH_1})^{-1}
& A^* (H_2 - \zeta^2 I_{\cH_2})^{-1}  \\  A (H_1 - \zeta^2 I_{\cH_1})^{-1}  &
\zeta (H_2 - \zeta^2 I_{\cH_2})^{-1}  \end{pmatrix},&    \\
\zeta^2 \in \rho(H_1) \cap \rho(H_2).&   \lb{B.25}
\end{split}
\end{align}
$(iii)$ In addition,
\begin{align}
\begin{split}
& \begin{pmatrix} f_1 \\ f_2 \end{pmatrix} \in \dom(Q) \, \text{ and } \,
Q \begin{pmatrix} f_1 \\ f_2 \end{pmatrix} = \eta \begin{pmatrix} f_1 \\ f_2 \end{pmatrix}, \;
\eta \neq 0,
 \\
& \quad \text{ implies } \, f_j \in \dom (H_j) \, \text{ and } \, H_j f_j = \eta^2 f_j, \; j=1,2.    \lb{B.26}
\end{split}
\end{align}
Conversely,
\begin{align}
\begin{split}
& f \in \dom(H_1) \, \text{ and } H_1 f = \lambda^2 f, \; \lambda \neq 0, \\
& \quad \text{implies } \, \begin{pmatrix} f \\ \lambda^{-1} A f \end{pmatrix} \in \dom(Q) \, \text{ and } \,
Q \begin{pmatrix} f \\ \lambda^{-1} A f \end{pmatrix}
= \lambda \begin{pmatrix} f \\ \lambda^{-1} A f \end{pmatrix}.
\end{split}
\end{align}
Similarly,
\begin{align}
\begin{split}
& g \in \dom(H_2) \, \text{ and } H_2 \, g = \mu^2 g, \; \mu \neq 0, \\
& \quad \text{implies } \, \begin{pmatrix} \mu^{-1} A^* g \\ g \end{pmatrix} \in \dom(Q) \, \text{ and } \,
Q \begin{pmatrix} \mu^{-1} A^* g \\ g \end{pmatrix}
= \mu \begin{pmatrix} \mu^{-1} A^* g \\ g \end{pmatrix}.
\end{split}
\end{align}
\end{theorem}

\section{Sesquilinear Forms and Associated Operators} \lb{sC}
\renewcommand{\theequation}{C.\arabic{equation}}
\renewcommand{\thetheorem}{C.\arabic{theorem}}
\setcounter{theorem}{0} \setcounter{equation}{0}

In this appendix we describe a few basic facts on sesquilinear forms and
linear operators associated with them following \cite[Sect.\ 2]{GM09}.
Let $\cH$ be a complex separable Hilbert space with scalar product
$(\dott,\dott)_{\cH}$ (antilinear in the first and linear in the second
argument), $\cV$ a reflexive Banach space continuously and densely embedded
into $\cH$. Then also $\cH$ embeds continuously and densely into $\cV^*$.
That is,
\begin{equation}
\cV  \hookrightarrow \cH  \hookrightarrow \cV^*.     \lb{C.1}
\end{equation}
Here the continuous embedding $\cH\hookrightarrow \cV^*$ is accomplished via
the identification
\begin{equation}
\cH \ni v \mapsto (\dott,v)_{\cH} \in \cV^*,     \lb{C.2}
\end{equation}
and recall our convention in this manuscript that if $X$ denotes a Banach space,
$X^*$ denotes the {\it adjoint space} of continuous  conjugate linear functionals on $X$,
also known as the {\it conjugate dual} of $X$.

In particular, if the sesquilinear form
\begin{equation}
{}_{\cV}\langle \dott, \dott \rangle_{\cV^*} \colon \cV \times \cV^* \to \bbC
\end{equation}
denotes the duality pairing between $\cV$ and $\cV^*$, then
\begin{equation}
{}_{\cV}\langle u,v\rangle_{\cV^*} = (u,v)_{\cH}, \quad u\in\cV, \;
v\in\cH\hookrightarrow\cV^*,   \lb{C.3}
\end{equation}
that is, the $\cV, \cV^*$ pairing
${}_{\cV}\langle \dott,\dott \rangle_{\cV^*}$ is compatible with the
scalar product $(\dott,\dott)_{\cH}$ in $\cH$.

Let $T \in\cB(\cV,\cV^*)$. Since $\cV$ is reflexive, $(\cV^*)^* = \cV$, one has
\begin{equation}
T \colon \cV \to \cV^*, \quad  T^* \colon \cV \to \cV^*   \lb{C.4}
\end{equation}
and
\begin{equation}
{}_{\cV}\langle u, Tv \rangle_{\cV^*}
= {}_{\cV^*}\langle T^* u, v\rangle_{(\cV^*)^*}
= {}_{\cV^*}\langle T^* u, v \rangle_{\cV}
= \ol{{}_{\cV}\langle v, T^* u \rangle_{\cV^*}}.
\end{equation}
{\it Self-adjointness} of $T$ is then defined by $T=T^*$, that is,
\begin{equation}
{}_{\cV}\langle u,T v \rangle_{\cV^*}
= {}_{\cV^*}\langle T u, v \rangle_{\cV}
= \ol{{}_{\cV}\langle v, T u \rangle_{\cV^*}}, \quad u, v \in \cV,    \lb{C.5}
\end{equation}
{\it nonnegativity} of $T$ is defined by
\begin{equation}
{}_{\cV}\langle u, T u \rangle_{\cV^*} \geq 0, \quad u \in \cV,    \lb{C.6}
\end{equation}
and {\it boundedness from below of $T$ by $c_T \in\bbR$} is defined by
\begin{equation}
{}_{\cV}\langle u, T u \rangle_{\cV^*} \geq c_T \|u\|^2_{\cH},
\quad u \in \cV.
\lb{C.6a}
\end{equation}
(By \eqref{C.3}, this is equivalent to
${}_{\cV}\langle u, T u \rangle_{\cV^*} \geq c_T \,
{}_{\cV}\langle u, u \rangle_{\cV^*}$, $u \in \cV$.)

Next, let the sesquilinear form $\ga(\dott,\dott)\colon\cV \times \cV \to \bbC$
(antilinear in the first and linear in the second argument) be
{\it $\cV$-bounded}, that is, there exists a $c_{\ga}>0$ such that
\begin{equation}
|\ga(u,v)| \le c_{\ga} \|u\|_{\cV} \|v\|_{\cV},  \quad u, v \in \cV.
\end{equation}
Then $\wti A$ defined by
\begin{equation}
\wti A \colon \begin{cases} \cV \to \cV^*, \\
\, v \mapsto \wti A v = \ga(\dott,v), \end{cases}    \lb{C.7}
\end{equation}
satisfies
\begin{equation}
\wti A \in\cB(\cV,\cV^*) \, \text{ and } \,
{}_{\cV}\big\langle u, \wti A v \big\rangle_{\cV^*}
= \ga(u,v), \quad  u, v \in \cV.    \lb{C.8}
\end{equation}
Assuming further that $\ga(\dott,\dott)$ is {\it symmetric}, that is,
\begin{equation}
\ga(u,v) = \ol{\ga(v,u)},  \quad u,v\in \cV,    \lb{C.9}
\end{equation}
and that $\ga$ is {\it $\cV$-coercive}, that is, there exists a constant
$C_0>0$ such that
\begin{equation}
\ga(u,u)  \geq C_0 \|u\|^2_{\cV}, \quad u\in\cV,    \lb{C.10}
\end{equation}
respectively, then,
\begin{equation}
\wti A \colon \cV \to \cV^* \, \text{ is bounded, self-adjoint, and boundedly
invertible.}    \lb{C.11}
\end{equation}
Moreover, denoting by $A$ the part of $\wti A$ in $\cH$ defined by
\begin{align}
\dom(A) = \big\{u\in\cV \,|\, \wti A u \in \cH \big\} \subseteq \cH, \quad
A= \wti A\big|_{\dom(A)}\colon \dom(A) \to \cH,   \lb{C.12}
\end{align}
then $A$ is a (possibly unbounded) self-adjoint operator in $\cH$ satisfying
\begin{align}
& A \geq C_0 I_{\cH},   \lb{C.13}  \\
& \dom\big(A^{1/2}\big) = \cV.  \lb{C.14}
\end{align}
In particular,
\begin{equation}
A^{-1} \in\cB(\cH).   \lb{C.15}
\end{equation}
The facts \eqref{C.1}--\eqref{C.15} are a consequence of the Lax--Milgram
theorem and the second representation theorem for symmetric sesquilinear forms.
Details can be found, for instance, in \cite[Sects.\ VI.3, VII.1]{DL00},
\cite[Ch.\ IV]{EE89}, and \cite{Li62}. 

Next, consider a symmetric form $\gb(\dott,\dott)\colon \cV\times\cV\to\bbC$
and assume that $\gb$ is {\it bounded from below by $c_{\gb}\in\bbR$}, that is,
\begin{equation}
\gb(u,u) \geq c_{\gb} \|u\|_{\cH}^2, \quad u\in\cV.  \lb{C.19}
\end{equation}
Introducing the scalar product
$(\dott,\dott)_{\cV_{\gb}^{}}\colon \cV\times\cV\to\bbC$
(and the associated norm $\|\cdot\|_{\cV_{\gb}^{}}$) by
\begin{equation}
(u,v)_{\cV_{\gb}^{}} = \gb(u,v) + (1- c_{\gb})(u,v)_{\cH}, \quad u,v\in\cV,  \lb{C.20}
\end{equation}
turns $\cV$ into a pre-Hilbert space $(\cV; (\dott,\dott)_{\cV_{\gb}^{}})$,
which we denote by
$\cV_{\gb}^{}$. The form $\gb$ is called {\it closed} in $\cH$ if $\cV_{\gb}^{}$ is actually
complete, and hence a Hilbert space. The form $\gb$ is called {\it closable}
in $\cH$ if it has a closed extension. If $\gb$ is closed in $\cH$, then
\begin{equation}
|\gb(u,v) + (1- c_{\gb})(u,v)_{\cH}| \le \|u\|_{\cV_{\gb}^{}} \|v\|_{\cV_{\gb}^{}},
\quad u,v\in \cV,
\lb{C.21}
\end{equation}
and
\begin{equation}
|\gb(u,u) + (1 - c_{\gb})\|u\|_{\cH}^2| = \|u\|_{\cV_{\gb}^{}}^2, \quad u \in \cV,
\lb{C.22}
\end{equation}
show that the form $\gb(\dott,\dott)+(1 - c_{\gb})(\dott,\dott)_{\cH}$ is a
symmetric, $\cV$-bounded, and $\cV$-coercive sesquilinear form. Hence,
by \eqref{C.7} and \eqref{C.8}, there exists a linear map
\begin{equation}
\wti B_{c_{\gb}} \colon \begin{cases} \cV_{\gb}^{} \to \cV_{\gb}^*, \\
\, v \mapsto \wti B_{c_{\gb}} v = \gb(\dott,v) +(1 - c_{\gb})(\dott,v)_{\cH},
\end{cases}
\lb{C.23}
\end{equation}
with
\begin{equation}
\wti B_{c_{\gb}} \in\cB(\cV_{\gb}^{},\cV_{\gb}^*) \, \text{ and } \,
{}_{\cV_{\gb}^{}}\big\langle u, \wti B_{c_{\gb}} v \big\rangle_{\cV_{\gb}^*}
= \gb(u,v)+(1 -c_{\gb})(u,v)_{\cH}, \quad  u, v \in \cV,    \lb{C.24}
\end{equation}
in particular, $\wti B_{c_{\gb}}$ is bounded, self-adjoint, and boundedly invertible. 
Introducing the linear map
\begin{equation}
\wti B = \wti B_{c_{\gb}} + (c_{\gb} - 1)\wti I \colon \cV_{\gb}^{}\to\cV_{\gb}^*,
\lb{C.24a}
\end{equation}
where $\wti I\colon \cV_{\gb}^{}\hookrightarrow\cV_{\gb}^*$ denotes
the continuous inclusion (embedding) map of $\cV_{\gb}^{}$ into $\cV_{\gb}^*$, 
$\wti B$ is bounded and self-adjoint, and one
obtains a self-adjoint operator $B$ in $\cH$ by restricting $\wti B$ to $\cH$,
\begin{align}
\dom(B) = \big\{u\in\cV \,\big|\, \wti B u \in \cH \big\} \subseteq \cH, \quad
B= \wti B\big|_{\dom(B)}\colon \dom(B) \to \cH,   \lb{C.25}
\end{align}
satisfying the following properties:
\begin{align}
& B \geq c_{\gb} I_{\cH},  \lb{C.26} \\
& \dom\big(|B|^{1/2}\big) = \dom\big((B - c_{\gb}I_{\cH})^{1/2}\big)
= \cV,  \lb{C.27} \\
& \gb(u,v) = \big(|B|^{1/2}u, U_B |B|^{1/2}v\big)_{\cH}    \lb{C.28b} \\
& \hspace*{.97cm}
= \big((B - c_{\gb}I_{\cH})^{1/2}u, (B - c_{\gb}I_{\cH})^{1/2}v\big)_{\cH}
+ c_{\gb} (u, v)_{\cH}
\lb{C.28} \\
& \hspace*{.97cm}
= {}_{\cV_{\gb}^{}}\big\langle u, \wti B v \big\rangle_{\cV_{\gb}^*},
\quad u, v \in \cV, \lb{C.28a} \\
& \gb(u,v) = (u, Bv)_{\cH}, \quad  u\in \cV, \; v \in\dom(B),  \lb{C.29} \\
& \dom(B) = \{v\in\cV\,|\, \text{there exists an $f_v\in\cH$ such that}  \no \\
& \hspace*{3.05cm} \gb(w,v)=(w,f_v)_{\cH} \text{ for all $w\in\cV$}\},
\lb{C.30} \\
& Bu = f_u, \quad u\in\dom(B),  \no \\
& \dom(B) \text{ is dense in $\cH$ and in $\cV_{\gb}^{}$}.  \lb{C.31}
\end{align}
Properties \eqref{C.30} and \eqref{C.31} uniquely determine $B$.
Here $U_B$ in \eqref{C.28} is the partial isometry in the polar
decomposition of $B$, that is,
\begin{equation}
B=U_B |B|, \quad  |B|=(B^*B)^{1/2} \geq 0.   \lb{C.32}
\end{equation}
The operator $B$ is called the {\it operator associated with the form $\gb$}.

The norm in the Hilbert space $\cV_{\gb}^*$ is given by
\begin{equation}
\|\ell\|_{\cV_{\gb}^*}
= \sup \{|{}_{\cV_{\gb}^{}}\langle u, \ell \rangle_{\cV_{\gb}^*}| \,|\,
\|u\|_{\cV_{\gb}^{}} \le 1\}, \quad \ell \in \cV_{\gb}^*,   \lb{C.34}
\end{equation}
with associated scalar product,
\begin{equation}
(\ell_1,\ell_2)_{\cV_{\gb}^*}
= {}_{\cV_{\gb}^{}}\big\langle
\big(\wti B+(1-c_{\gb})\wti I\,\big)^{-1}
\ell_1, \ell_2 \big\rangle_{\cV_{\gb}^*},
\quad \ell_1, \ell_2 \in \cV_{\gb}^*.   \lb{C.35}
\end{equation}
Since
\begin{equation}
\big\|\big(\wti B + (1 - c_{\gb})\wti I\,\big)v \big\|_{\cV_{\gb}^*}
= \|v\|_{\cV_{\gb}^{}}, \quad v\in\cV,   \lb{C.36}
\end{equation}
the Riesz representation theorem yields
\begin{equation}
\big(\wti B+(1-c_{\gb})\wti I\,\big) \in \cB(\cV_{\gb}^{},\cV_{\gb}^*) \,
\text{ and }\big(\wti B + (1 - c_{\gb})\wti I\,\big) \colon \cV_{\gb}^{}
\to \cV_{\gb}^* \, \text{ is unitary.}   \lb{C.37}
\end{equation}
In addition,
\begin{align} 
{}_{\cV_{\gb}^{}}\big\langle u,\big(\wti B
+ (1 - c_{\gb})\wti I\,\big) v \big\rangle_{\cV_{\gb}^*}
& = \big((B+(1-c_{\gb})I_{\cH})^{1/2}u,
(B+(1-c_{\gb})I_{\cH})^{1/2}v \big)_{\cH}    \no \\
& = (u,v)_{\cV_{\gb}^{}},  \quad  u, v \in \cV_{\gb}^{}.   \lb{C.38}
\end{align}
In particular,
\begin{equation}
\big\|(B+(1-c_{\gb})I_{\cH})^{1/2}u\big\|_{\cH} = \|u\|_{\cV_{\gb}^{}},
\quad u \in \cV_{\gb}^{},   \lb{C.39}
\end{equation}
and hence 
\begin{equation}
(B+(1-c_{\gb})I_{\cH})^{1/2}\in\cB(\cV_{\gb}^{},\cH) \, \text{ and }
(B + (1 - c_{\gb})I_{\cH})^{1/2} \colon \cV_{\gb}^{} \to \cH \, \text{ is unitary.}
\lb{C.40}
\end{equation}
The facts \eqref{C.19}--\eqref{C.40} comprise the second representation
theorem of sesquilinear forms (cf.\ \cite[Sect.\ IV.2]{EE89},
\cite[Sects.\ 1.2--1.5]{Fa75}, and \cite[Sect.\ VI.2.6]{Ka80}).

We briefly supplement \eqref{C.19}--\eqref{C.40} with some considerations that 
hint at mapping properties of $\big(\wti B+(1-c_{\gb})\wti I\,\big)^{\pm 1/2}$ on a scale of spaces,  
which, for simplicity, we restrict to the triple of spaces $\cV_{\gb}^{}$, $\cH$, and $\cV_{\gb}^*$ in this appendix. We start by defining 
\begin{equation}
\big(\hat B_{c_{\gb}} + (1-c_{\gb}) \hat I\big)^{1/2} \colon \begin{cases} \cV_{\gb}^{} \to \cH, \\
\, v \mapsto (B + (1 - c_{\gb}) I_{\cH})^{1/2} v,
\end{cases}        \lb{C.41}
\end{equation}
and similarly,
\begin{equation}
\big(\check B_{c_{\gb}} + (1-c_{\gb}) \check I\big)^{1/2} \colon \begin{cases} \cH \to \cV_{\gb}^*, \\
\, f \mapsto \gb \big(\dott, (B + (1 - c_{\gb}) I_{\cH})^{-1/2} f\big) \\
\qquad \; + (1 - c_{\gb}) \big(\dott, (B + (1 - c_{\gb}) I_{\cH})^{-1/2} f \big)_{\cH}. 
\end{cases}         \lb{C.41a}
\end{equation}
Then both maps in \eqref{C.41} and \eqref{C.41a} are bounded and boundedly invertible. 
In particular,
\begin{align}
\begin{split} 
&\big(\hat B_{c_{\gb}} + (1-c_{\gb}) \hat I\big)^{1/2} \in \cB(\cV_{\gb}^{},\cH), \quad 
\big(\hat B_{c_{\gb}} + (1-c_{\gb}) \hat I\big)^{-1/2} \in \cB(\cH,\cV_{\gb}^{}),     \lb{C.41b} \\
&\big(\check B_{c_{\gb}} + (1-c_{\gb}) \check I\big)^{1/2} \in \cB(\cH,\cV_{\gb}^*),  \quad 
\big(\check B_{c_{\gb}} + (1-c_{\gb}) \check I\big)^{-1/2} \in \cB(\cV_{\gb}^*,\cH),    
\end{split} 
\end{align}
and
\begin{align}
& \big(\hat B_{c_{\gb}} + (1-c_{\gb}) \hat I\big)^{1/2} 
\big(\check B_{c_{\gb}} + (1-c_{\gb}) \check I\big)^{1/2} 
= \big(\wti B+(1-c_{\gb})\wti I\,\big) \in \cB(\cV_{\gb}^{},\cV_{\gb}^*),    \no \\
& \big(\check B_{c_{\gb}} + (1-c_{\gb}) \check I\big)^{-1/2} 
\big(\hat B_{c_{\gb}} + (1-c_{\gb}) \hat I\big)^{-1/2} 
= \big(\wti B+(1-c_{\gb})\wti I\,\big)^{-1} \in \cB(\cV_{\gb}^*,\cV_{\gb}^{}).   \lb{C.41c}
\end{align}
Due to self-adjointness of $\wti B$ as a bounded map from $\cV_{\gb}^{}$ to $\cV_{\gb}^*$ 
in the sense of \eqref{C.5}, one finally obtains that
\begin{align}
\begin{split} 
&\Big(\big(\hat B_{c_{\gb}} + (1-c_{\gb}) \hat I\big)^{\pm 1/2}\Big)^* 
= \big(\check B_{c_{\gb}} + (1-c_{\gb}) \check I\big)^{\pm 1/2},   \\
& \Big(\big(\check B_{c_{\gb}} + (1-c_{\gb}) \check I\big)^{\pm 1/2}\Big)^* 
= \big(\hat B_{c_{\gb}} + (1-c_{\gb}) \hat I\big)^{\pm 1/2}.    \lb{C.41d}
\end{split} 
\end{align}
Hence, we will follow standard practice in connection with chains of 
(Sobolev) spaces and refrain from painstakingly distinguishing the\, $\hat{}$ - and\, 
$\check{}$ -operations and simply resort to the notation
\begin{equation}
\big(\wti B+(1-c_{\gb})\wti I\,\big)^{\pm 1/2}      \lb{C.41e} 
\end{equation}
for the operators in \eqref{C.41b} in the bulk of this paper. 

A special but important case of nonnegative closed forms is obtained as
follows: Let $\cH_j$, $j=1,2$, be complex separable Hilbert spaces, and
$T\colon \dom(T)\to\cH_2$, $\dom(T)\subseteq \cH_1$, a densely defined
operator. Consider the nonnegative form
$\ga_T\colon \dom(T)\times \dom(T)\to\bbC$ defined by
\begin{equation}
\ga_T(u,v)=(Tu,Tv)_{\cH_2}, \quad u, v \in\dom(T).   \lb{C.42}
\end{equation}
Then the form $\ga_T$ is closed (resp., closable) in $\cH_1$ if and only if $T$ is.
If $T$ is closed, the unique nonnegative self-adjoint operator associated
with $\ga_T$ in $\cH_1$, whose existence is guaranteed by the second
representation theorem for forms, then equals $T^*T\geq 0$. In particular,
one obtains in addition to \eqref{C.42},
\begin{equation}
\ga_T(u,v) = (|T|u,|T|v)_{\cH_1}, \quad u, v \in\dom(T)=\dom(|T|).   \lb{C.43}
\end{equation}
Moreover, since
\begin{align}
\begin{split}
& \gb(u,v) +(1-c_{\gb})(u,v)_{\cH}
= \big((B+(1-c_{\gb})I_{\cH})^{1/2}u, (B+(1-c_{\gb})I_{\cH})^{1/2}v\big)_{\cH}, \\
& \hspace*{6.1cm}  u, v \in \dom(b) = \dom\big(|B|^{1/2}\big)=\cV,   \lb{C.43a}
\end{split}
\end{align}
and $(B+(1-c_{\gb})I_{\cH})^{1/2}$ is self-adjoint (and hence closed) in
$\cH$, a symmetric, $\cV$-bounded, and $\cV$-coercive form is densely
defined in $\cH\times\cH$ and closed in $\cH$ (a fact we will be using in the proof of
Theorem \ref{t2.3}). We refer to \cite[Sect.\ VI.2.4]{Ka80} and
\cite[Sect.\ 5.5]{We80} for details.

Next we recall that if $\ga_j$ are sesquilinear forms defined on
$\dom(\ga_j)$, $j=1,2$, bounded from below and closed, then also
\begin{equation}
(\ga_1 + \ga_2)\colon \begin{cases} (\dom(\ga_1)\cap\dom(\ga_2))\times
(\dom(\ga_1)\cap\dom(\ga_2)) \to \bbC, \\
(u,v) \mapsto (\ga_1 + \ga_2)(u,v) = \ga_1(u,v) + \ga_2(u,v) \end{cases}   \lb{C.44}
\end{equation}
is bounded from below and closed (cf., e.g., \cite[Sect.\ VI.1.6]{Ka80}).

Finally, we also recall the following perturbation theoretic fact:
Suppose $\ga$ is a sesquilinear form defined on $\cV\times\cV$, bounded
from below and closed, and let $\gb$ be a symmetric sesquilinear form
bounded with respect to $\ga$ with bound less than one, that is,
$\dom(\gb)\supseteq \cV\times\cV$, and that there exist $0\le \alpha < 1$ and
$\beta\ge 0$ such that
\begin{equation}
|\gb(u,u)| \le \alpha |\ga(u,u)| + \beta \|u\|_{\cH}^2, \quad u\in \cV.   \lb{C.45}
\end{equation}
Then
\begin{equation}
(\ga + \gb)\colon \begin{cases} \cV\times\cV \to \bbC, \\
\hspace*{.12cm} (u,v) \mapsto (\ga + \gb)(u,v) = \ga(u,v) + \gb(u,v) \end{cases}
\lb{C.46}
\end{equation}
defines a sesquilinear form that is bounded from below and closed
(cf., e.g., \cite[Sect.\,VI.1.6]{Ka80}). In the special case where $\alpha$
can be chosen arbitrarily small, the form $\gb$ is called infinitesimally
form bounded with respect to $\ga$.

\medskip

\noindent {\bf Acknowledgments.}
We gratefully acknowledge valuable correspondence with Rostyslav Hryniv,
Mark Malamud, Roger Nichols, Fritz Philipp, Barry Simon, G\"unter Stolz, and 
Gerald Teschl. In addition, we are indebted to Igor Verbitsky for helpful discussions.


\end{document}